\documentclass{amsart}
\usepackage{amsmath,amsthm}
\usepackage{amsfonts,amssymb}
\usepackage{accents}
\usepackage{enumerate}
\usepackage{accents,color}
\usepackage{graphicx}
\usepackage{wrapfig}
\newcommand{\lvt}{\left|\kern-1.35pt\left|\kern-1.3pt\left|}
\newcommand{\rvt}{\right|\kern-1.3pt\right|\kern-1.35pt\right|}

\makeatletter
\def\@tocline#1#2#3#4#5#6#7{\relax
  \ifnum #1>\c@tocdepth 
  \else
    \par \addpenalty\@secpenalty\addvspace{#2}%
        \begingroup \hyphenpenalty\@M
    \@ifempty{#4}{%
      \@tempdima\csname r@tocindent\number#1\endcsname\relax
    }{%
      \@tempdima#4\relax
    }%
    \parindent\z@ \leftskip#3\relax \advance\leftskip\@tempdima\relax
    \rightskip\@pnumwidth plus4em \parfillskip-\@pnumwidth
    #5\leavevmode\hskip-\@tempdima #6\nobreak\relax
    \ifnum #1=0
    \hfil\hbox to\@pnumwidth{}
    \else
    \hfil\hbox to\@pnumwidth{\@tocpagenum{#7}}\fi
    \par
    \nobreak
    \endgroup
  \fi}
\makeatother

\newtheorem{thm}{Theorem}[section]
\newtheorem{cor}[thm]{Corollary}
\newtheorem{lem}[thm]{Lemma}
\newtheorem{prop}[thm]{Proposition}

\newtheorem{defn}[thm]{Definition}
       
\theoremstyle{remark}
\newtheorem{rem}{Remark}[section]

 \def\la{{\langle}}
 \def\ra{{\rangle}}
 \def\ve{{\varepsilon}}
\def\({\left(}
\def \){ \right)}
\def\[{\left[}
\def \]{ \right]}

 \def\d{\mathrm{d}}
 
 \def\sph{{\mathbb{S}^{d-1}}}

 \def\sC{{\mathsf C}}
 
 \def\sE{{\mathsf E}}

 \def\sK{{\mathsf K}}
 \def\sL{{\mathsf L}}
 
 \def\sP{{\mathsf P}}
 
 \def\sS{{\mathsf S}}
 
 \def\bs{{\mathsf b}}
 \def\sc{{\mathsf c}}
 \def\sd{{\mathsf d}}
 \def\sm{{\mathsf m}}
 \def\sw{{\mathsf w}}

 \def\fD{{\mathfrak D}}
 
 \def\a{{\alpha}}
 \def\b{{\beta}}
 \def\g{{\gamma}}
 \def\k{{\kappa}}
 \def\t{{\theta}}
 \def\l{{\lambda}}
 \def\o{{\omega}}
 \def\s{\sigma}
 \def\la{{\langle}}
 \def\ra{{\rangle}}
 \def\ve{{\varepsilon}}
 
 \def\bb{{\mathbf b}}
 \def\cb{{\mathbf c}}

 \def\kb{{\mathbf k}}
 
 \def\mb{{\mathbf m}}

 \def\Eb{{\mathbf E}}

 \def\Kb{{\mathbf K}}
 \def\Lb{{\mathbf L}}
 \def\Pb{{\mathbf P}}
 
 \def\Sb{{\mathbf S}}

 \def\CH{{\mathcal H}}

 \def\CT{{\mathcal T}}
 
 \def\CV{{\mathcal V}}

 \def\BB{{\mathbb B}}

 \def\NN{{\mathbb N}}

 \def\RR{{\mathbb R}}
 \def\SS{{\mathbb S}}
 
 \def\VV{{\mathbb V}}
 \def\XX{{\mathbb X}}
 
      \def\proj{\operatorname{proj}}

\def\lla{\langle{\kern-2.5pt}\langle}      
\def\rra{\rangle{\kern-2.5pt}\rangle}

\newcommand{\wh}{\widehat}

\def\f{\frac}

\graphicspath{{./}}

\begin{document}
 
\title[Approximation and localized polynomial frame]
{Approximation and localized polynomial frame on double hyperbolic and conic domains}

\author{Yuan~Xu}
\address{Department of Mathematics, University of Oregon, Eugene, 
OR 97403--1222, USA}
\email{yuan@uoregon.edu} 

\date{\today}  
\subjclass[2010]{41A10, 41A63, 42C10, 42C40}
\keywords{Approximation, conic domain, Fourier orthogonal series, homogeneous spaces, localized kernel, 
orthogonal polynomials, polynomial frame}

\begin{abstract}
We study approximation and localized polynomial frames on a bounded double hyperbolic or conic surface
and the domain bounded by such a surface and hyperplanes. The main work follows the framework 
developed recently in \cite{X21} for homogeneous spaces that are assumed to contain 
highly localized kernels constructed via a family of orthogonal polynomials. The existence of such kernels
will be established with the help of closed form formulas for the reproducing kernels. The main 
results provide a construction of semi-discrete localized tight frame in weighted $L^2$ norm and a 
characterization of best approximation by polynomials on our domains. Several intermediate results, including
the Marcinkiewicz-Zygmund inequalities, positive cubature rules, Christoeffel functions, and Bernstein type 
inequalities, are shown to hold for doubling weights defined via the intrinsic distance on the domain.
\end{abstract}
 
\maketitle
 
\section{Introduction}
\setcounter{equation}{0}

Recently a general framework based on highly localized polynomial kernels is developed in \cite{X21} for
localizable homogeneous spaces and used for studying approximation and localized polynomial frame on 
a finite conic surface and the domain bounded by such a surface and a hyperplane. In the present work, 
we establish highly localized kernels and carry out analysis on bounded double hyperbolic or conic surface 
and the domain bounded by such a surface and hyperplanes. 

\subsection{Motivation}
Let $\Omega$ be a set in $\RR^d$, either an algebraic surface or a domain with non-empty interior. A 
homogeneous space is a measure space $(\Omega, \varpi, \sd)$, where $\varpi$ is a nonnegative 
doubling weight function with respect to the metric $\sd(\cdot,\cdot)$ on $\Omega$. We call a homogeneous 
space $(\Omega, \varpi, \sd)$ {\it localizable} if it contains highly localized kernels. 

The kernels are constructed using orthogonal polynomials. Let $\d \sm$ be the Lebesgue measure on $\Omega$. 
Assume that the weight function $\varpi$ is regular so that 
\begin{equation}\label{def:ipd}
    \la f, g\ra_{\varpi} := \int_{\Omega} f(x) g(x) \varpi(x) \d \sm(x)
\end{equation}   
is a well defined inner product on the space of polynomials restricted to $\Omega$. Let $\CV_n(\Omega, \varpi)$ 
be the space of orthogonal polynomials of degree $n$ with respect to this inner product. The projection
operator $\proj_n:  L^2(\Omega, \varpi)\mapsto \CV_n(\Omega,\varpi)$ can be written as 
\begin{equation}\label{def:projPn}
  \proj_n(\varpi; f, x) = \int_\Omega f(y) P_n(\varpi; x,y) \varpi(y) \d \sm(y), \quad f\in L^2(\Omega, \varpi).
\end{equation}
where $P_n(\varpi; \cdot,\cdot)$ is the reproducing kernel of the space $\CV_n(\Omega, \varpi)$. Let 
$\wh a$ be a cut-off function, defined as a compactly supported function in $C^\infty(\RR_+)$. Then  
our highly localized kernels are of the form 
\begin{equation*} 
L_n(\varpi;x,y) := \sum_{j=0}^\infty \wh a \Big(\frac{j}{n}\Big) P_j(\varpi; x,y).
\end{equation*}
The kernel is highly localized if it decays at rates faster than any inverse polynomial rate away from the 
main diagonal $y=x$ in $\Omega \times \Omega$ with respect  to the distance $\sd$ on $\Omega$; see
the definition in the next section. These kernels provide important tools for analysis on regular domains,
such as the unit sphere and the unit ball, and are essential ingredient in recent study of approximation 
and localized polynomial frames; see, for example, \cite{BD, DaiX, KPX2, IPX2, NPW1, NPW2, PX2} 
for some of the results on the spheres and balls and \cite{BKMP1, BKMP2, KP1, KP2, LSWW, WLSW} 
for various applications. 

The reason that highly localized kernels are known only on a few regular domains lies in the {\it addition 
formula} for orthogonal polynomials, which are closed form formulas for the reproducing kernels of 
orthogonal polynomials. For spherical harmonics that are orthogonal on the unit sphere, which serve as 
a quintessential example for our study, the closed form formula is given by $Z_n(\la x, y\ra)$, where $Z_n$ 
is a Gegenbauer polynomial of one variable. The addition formulas are powerful tools when they exist.

Our recent work in \cite{X21} is prompted by two aspirations. The first one is the realization that, if highly 
localized kernels are taken as granted, then much of the analysis can be developed within a general 
framework of homogeneous spaces. The second one is the possibility of establishing highly localized kernels 
and carrying out analysis thereafter on conic domains, which are domains largely untouched hitherto. The
latter is made possible by recently discovered new addition formulas for orthogonal polynomials on conic 
domains \cite{X20a, X20b}. Altogether there are four types of of conic domains and they are standardized as: 
\begin{enumerate}[\quad (1)]
 \item conic surface $\VV_0^{d+1}$ defined by 
$$
   \VV_0^{d+1}: = \left\{(x,t) \in \RR^{d+1}: \|x\| = t, \, 0 \le t \le 1, \, x\in \RR^d\right\};
$$
\item solid cone $\VV^{d+1}$ bounded by $\VV_0^{d+1}$ and the hyperplane $t =1$;
\item two-sheets hyperbolic surface ${}_\varrho \XX_0^{d+1}$ defined by 
$$
{}_\varrho \XX_0^{d+1}: = \left\{(x,t) \in \RR^{d+1}: \|x\|^2 = t^2-\varrho^2, \, 
         \varrho \le |t| \le 1+ \varrho, \, x\in \RR^d\right\}, 
$$
where $\varrho \ge 0$, which becomes the double conic surface $\XX_0^{d+1}$ when $\varrho =0$;
\item solid hyperboloid ${}_\varrho \XX^{d+1}$ bounded by ${}_\varrho \XX_0^{d+1}$ and the hyperplanes 
$t = \pm 1$, which becomes solid double cone $\XX^{d+1}$ when $\varrho =0$. 
\end{enumerate}
The first two types of conic domains are studied in \cite{X21}. The present paper deals with the other two
cases: double hyperbolic surfaces and hyperboloid. 

\subsection{Main results}
We will follow the framework established in \cite{X21}. Assuming several assertions on highly localized 
kernels and on fast decaying polynomials, the framework provides a unified theory for two objectives. 
The first one leads to localized polynomial frames constructed via a semi-continuous Calder\'on type decomposition 
$$
  f  = \sum_{j=0}^\infty L_{2^j} \ast L_{2^j} \ast f, \qquad  f\in L^2(\Omega, \varpi), 
$$
where $L_n*f$ denotes the integral operator that has $L_n(\varpi; \cdot,\cdot)$ as its kernel,
\begin{equation} \label{def:Ln-operator}
   L_n * f (x): = \int_\Omega f(y)L_n(\varpi; x,y) \varpi(y) \d y. 
\end{equation}
Discretizing the integrals by appropriate positive cubature rules, we end up with a fame system $\{\psi_\xi\}_{\xi \in \Xi}$,
indexed by a discrete set $\Xi$ of well-separated points in $\Omega$, where $\psi_\xi(x) = 
\sqrt{\l_{\xi}} L_{2^j} (\varpi; x,\xi)$ with $\l_{\xi} >0$ being the coefficients in the cubature, and the frame is tight in
the sense that, for all $f \in L^2(\Omega, \varpi)$, 
$$
  f = \sum_{\xi \in \Xi} \la f,\psi_\xi\ra_{\varpi} \psi_\varpi \quad \hbox{and} \quad 
   \int_\Omega |f(x)|^2 \varpi(x) \d x = \sum_{\xi \in \Xi} |\la f, \psi_\xi\ra|^2.
$$
This is an extension of the extensive work on the unit sphere and the unit ball in the literature that we have
described in the previous subsection. The second objective is to study the error of the best approximation by 
polynomials
$$
  E_n (f)_{L^p(\Omega,\varpi)} = \inf_{\deg g \le n} \| f - g\|_{L^p(\Omega,\varpi)}.
$$
The aim is to provide a characterization via a modulus of smoothness, defined as a multiplier operator, 
and an equivalent $K$-functional, defined by the spectral differential operator that has orthogonal polynomials as 
eigenfunctions. Such a characterization is in line with those on the unit sphere, see \cite{DaiX, Rus,X05} and
the references therein. The characterization consists of a direct estimate, using the fast decaying of 
$L_n(\varpi; \cdot,\cdot)$, and an inverse estimate, using a Bernstein inequality established via
highly localized kernels. It is worth mentioning that several intermediate results, such as the 
Marcinkiewicz-Zygmund inequalities, positive cubature rules, Christoeffel functions, and Bernstein type 
are of independent interests; furthermore, some of these intermediate results can be established 
for doubling weights, which extends the results on the interval and on the unit sphere \cite{Dai1, DaiX, 
MT}. 

In order for the framework to work on a particular domain, substantial work is needed to fulfill the 
assertions and assumptions, which depends heavily on the geometry of the domain and the complexity 
of the orthogonal structure on the domain. For each conic domain, this starts with identifying an appropriate 
intrinsic distance on the domain. Establishing the highly localized kernels is delicate and fairly involved, 
relying on the structure of addition formula and the distance function. The $\ve$-separated set is assumed 
conceptually in the framework, an explicit construction is needed on each domain, which is also necessary 
for computational purpose. Even though the roadmap is outlined by the framework, the actual work on each 
domain remains challenging and amounts to a thorough understanding of the intrinsic structure on the domain. 

It should also be mentioned that the structures on double hyperbolic domains $\XX_0^{d+1}$ and
$\XX^{d+1}$, which degenerate to double cones when $\varrho= 0$, are distinctively different from those
 on the single conic domains $\VV_0^{d+1}$ and $\VV^{d+1}$. Indeed, the distance functions for these 
 two typos of domains are incomparable, around the conic apex, and the orthogonal structure on the double 
 conic domains depends on the parity of the polynomials on the the variable $t$. For example, for the 
 double conic surface $\XX_0^{d+1}$, the closed form formula for the reproducing kernels is stablished 
 for polynomials that are orthogonal with respect to $\sw_{\b,\g} (t) = |t|^{2\b}(1-t^2)^{\g-\f12}$, but only for 
 the subspace of polynomials that are even in $t$ variable or odd in $t$ variable. The restriction requires us 
 to restrict to the class of functions that are even in $t$ variable, but the entire framework can still be 
 carried out on the domain. 
Using the reproducing kernels in \cite{X20b}, we will show that $(\XX_0^{d+1}, \sw_{0,\g}, \sd_{\XX_0})$ is 
a localizable homogeneous space. Similar result holds for the solid double cone $\XX^{d+1}$ with the weight 
function $W_{\b,\g,\mu}(x,t) = |t|^\b (1-t)^\g (t^2-\|x\|^2)^{\mu-\f12}$ and its own distance function. For these
homogeneous spaces we shall show that the framework in \cite{X21} on localizable tight frame and on the 
best approximation can be carried out completely. 
 
\subsection{Organization and convention}

The general framework on the localizable homogeneous space emerges from the study on the unit sphere 
and the unit ball, as explained in \cite{X21}, and it consists of several topics. Because of its length, we shall 
refer to its discussion and formulation to \cite{X21}. The present work, however, is self-contained otherwise 
and can be read independently.  

The paper is organized as follows. We will state the assertions and just enough background to state them 
for the general framework in the next section. Whenever possible, we will not give definition and properties 
for homogeneous spaces but only on conic domains and only when they are needed. For a more extensive 
discussion on the background materials, we also refer to \cite{X21}. The cases of hyperbolic surface and 
hyperboloid will be studied in Section 3 and Section 4, respectively, and the two sections will have parallel 
structures, which is also the structure for the conic surface and the cone in \cite{X21}, to emphasis what 
is required to carry out the program and for easier comparison between the cases. Each section will contain 
several sections and its organization will be described in the preamble of the section.  

Throughout this paper, we will denote by $L^p(\Omega, \sw)$ the weighted $L^p$ space with respect to
the weight function $\sw$ defined on the domain $\Omega$ for $1 \le p \le \infty$. Its norm will be denote 
by $\|\cdot\|_{p,\sw}$ for $1 \le p \le \infty$ with the understanding that the space is $C(\Omega)$ with 
the uniform norm when $p= \infty$. 

Finally, we shall use $c$, $c'$, $c_1$, $c_2$ etc. to denote positive constants that depend on fixed parameters and 
their values may change from line to line. Furthermore, we shall write $A \sim B$ if $ c' A \le B \le c A$. 

\section{Preliminaries}
\setcounter{equation}{0}

In the first subsection we collect all assertions needed for the general framework on the localized homogeneous
space. Basics of classical orthogonal polynomials that will be used latter are collected in the second subsection. 

\subsection{Assertions for the general framework}
Let $(\Omega, \varpi, \sd)$ be a homogeneous space. For a given set $E \subset \Omega$, we define 
$
      \varpi(E) = \int_{E} \varpi(x) \d \sm(x).
$
The weight function $\varpi$ is a doubling weight if there exists a constant $L > 0$ such that 
$$
   \varpi(B(x,2 r)) \le L \varpi(B(x,r)), \qquad \forall x \in \Omega, \quad r \in (0, r_0), 
$$
where $r_0$ is the largest positive number such that $B(x,r)=  \{y \in \Omega: \sd(x,t) < r\} \subset \Omega$.
The smallest $L$ for doubling inequality to hold is called the doubling constant of $\varpi$. 

Let $\varpi$ be a nonnegative weight function on $\Omega$ and let $\la\cdot,\cdot\ra_\varpi$ be the 
inner product defined by \eqref{def:ipd}. For $n =0,1,2,\ldots$, let $\CV_n(\Omega, \varpi)$ be the space 
of orthogonal polynomials with respect to the inner product. If $\{P_{\nu,n}: 1 \le \nu \le \dim \CV_n(\Omega,\varpi)\}$ 
is an orthogonal basis of $\CV_n(\Omega,\varpi)$, then the reproducing kernel of $\CV_n(\Omega,\varpi)$ 
satisfies 
\begin{equation}\label{eq:reprod-kernel}
P_n(\varpi; x, y) = \sum_{\nu=1}^{\dim\CV_n^d(\Omega,\varpi)} \frac{P_{\nu,n}(x) P_{\nu,n}(y)}{\la P_{\nu,n},P_{\nu,n} \ra_\varpi}.
\end{equation}
It is the kernel of the projection operator $\proj_n: L^2(\Omega, \varpi)\mapsto \CV_n^d(\Omega, \varpi)$
stated in \eqref{def:projPn}. The Fourier orthogonal series of $f \in L^2(\Omega, \varpi)$ is given by
\begin{equation}\label{eq:fourier-series}
  f = \sum_{n=0}^\infty \proj_n(\varpi; x,y).
\end{equation}

\subsubsection{Highly localized kernels}
A nonnegative function $\wh a \in C^\infty(\RR)$ is said to be admissible if it obeys either one of the conditions
\begin{enumerate}[         $(a)$]
\item $\mathrm{supp}\,  \wh a \subset [0, 2]$ and $\wh a(t) = 1$, $t\in [0, 1]$; or
\item $\mathrm{supp}\,  \wh a \subset [1/2, 2]$. 
\end{enumerate}
Given such a cut-off function, we define a kernel $L_n(\varpi;\cdot,\cdot)$ on $\Omega\times\Omega$ by
\begin{equation*} 
L_n(\varpi;x,y) := \sum_{j=0}^\infty \wh a \Big(\frac{j}{n}\Big) P_j(\varpi; x,y).
\end{equation*}

\begin{defn}\label{def:Assertions}
The kernels $L_n(\varpi; \cdot,\cdot)$, $n=1,2,\ldots$, are called highly localized if they satisfy the following 
assertions: 
\begin{enumerate}[\,  \bf 1]
\item[] \textbf{Assertion 1}. For $\k > 0$ and $x, y\in \Omega$, 
$$
   |L_n(\varpi; x ,y)| \le c_\k \frac{1} {\sqrt{\varpi\!\left(B(x,n^{-1})\right)} \sqrt{\varpi\!\left(B(x,n^{-1})\right)}
      \left(1+n \sd(x,y)\right)^\k}.
$$
\medskip\noindent
\item[] 
\textbf{Assertion 2}.  For $0 < \delta \le \delta_0$ with some $\delta_0<1$ and $x_1 \in B(x_2, \frac{\delta}{n})$, 
$$
   \left|L_n(\varpi; x_1,y) - L_n(\varpi; x_2,y)\right| \le c_\k \frac{n \sd(x_1,x_2)} 
        {\sqrt{\varpi\!\left(B(x_1,n^{-1})\right)} \sqrt{\varpi\!\left(B(x_2,n^{-1})\right)} \left(1+n \sd(x_2,y)\right)^\k}.
$$
\item[]
\textbf{Assertion 3}. For sufficient large $\k >0$, 
\begin{align*}
\int_{\Omega} \frac{ \varpi(y)}{\varpi\!\left(B(y,n^{-1})\right)
    \big(1 + n \sd(x,y) \big)^{\k}}    \d \sm(y) \le c.
\end{align*}
\end{enumerate}
\end{defn}

The third assertion affirms the sharpness of the first two assertions, as can be seen from the following 
inequality \cite[Lemma 2.1.4]{X21}.

\begin{lem}\label{lem:CorA3}
Let $\varpi$ be a doubling weight that satisfies Assertion 3 with $\k >0$. For $ 0< p < \infty$, let 
$\tau =\k - \f p 2 \a(\varpi) |1-\f{p}{2}| > 0$. Then, for $x\in \Omega$, 
\begin{equation}\label{eq:CorA3}
\int_{\Omega} \frac{ \varpi(y)}{\varpi\!\left(B\left(y,n^{-1}\right)\right)^{\f p 2} \big(1 + n \sd(x,y) \big)^{\tau}} \d \sm(y) \le 
   c \,  \varpi\!\left(B\left(x,n^{-1}\right)\right)^{1-\f p 2}.
\end{equation}
In particular,  the highly localized kernel satisfies
\begin{equation}\label{eq:Ln-bdd}
   \int_{\Omega} \left|L_n (\varpi;x,y) \right|^p \varpi(s) \d \sm(y)
       \le c \left[ \varpi\!\left(B\!\left(x,n^{-1}\right)\right)\right]^{1-p}. 
\end{equation}
\end{lem}
 
The homogeneous space $(\Omega, \varpi, \sd)$ is called {\it localizable} if $\varpi$ is a doubling weight that 
admits highly localized kernels. 

\subsubsection{Well-separated set of points}
We will need well distributed points for discretization, such as in the Marcinkiewicz-Zygmund inequality 
and the positive cubature rules, and in our localized polynomial frame. The precise definition is given 
as follows. 
 
\begin{defn}\label{defn:separated-pts}
Let $\Xi$ be a discrete set in $\Omega$. 
\begin{enumerate} [  \quad (a)]
\item A finite collection of subsets $\{S_z: z \in \Xi\}$ is called a partition of $\Omega$ if $S_z^\circ\cap 
S_y^\circ  = \emptyset$ when $z \ne y$ and $\Omega = \bigcup_{z \in \Xi} S_z$. 
\item Let $\ve>0$. A discrete subset $\Xi$ of $\Omega$ is called $\ve$-separated if $\sd(x,y) \ge\ve$
for every two distinct points $x, y \in \Xi$. 
\item $\Xi$ is called maximal if there is a constant $c_d > 1$ such that 
\begin{equation}\label{eq:def-pts2}
  1 \le  \sum_{z\in \Xi} \chi_{B(z, \ve)}(x) \le c_d, \qquad \forall x \in \Omega,
\end{equation}
where $\chi_E$ denotes the characteristic function of the set $E$.
\end{enumerate}
\end{defn} 

The existence of $\ve$-separated points is assumed in the general framework for homogeneous spaces
in \cite{X21}. The construction of such points on a given domain depends on the geometry and the distance 
function, and it is crucial for explicit formulation and practical computation of localized frames. For conic 
domains $\XX_0^{d+1}$ and $\XX^{d+1}$, we will give a construction of such points in latter sections. 

\subsubsection{Fast decaying polynomials and cubature rules}
Let $\sw$ be a doubling weight on $\Omega$. A cubature rule is a finite sum that discretize a given integral,  
$$
  \int_{\Omega} f(z) \sw (x) \d x = \sum_{k=1}^N \l_k f(x_k),
$$
where the equality holds for a given polynomials subspace. We will need positive cubature rules, for which 
all $\l_k$ are positive, over $\ve$-separated points. To quantify the coefficients $\l_k$, which are used for 
our localized polynomial frame, we need to show that the Christoffel function $\l_n(\sw;\cdot)$, defined by 
\begin{align}\label{eq:ChristoffelF}
   \l_n(\sw;x): = \inf_{\substack{g(x) =1 \\ g \in \Pi_n(\Omega)}} \int_{\Omega} |g(x)|^2 \sw(x)  \d \sm(x),
\end{align} 
are bounded, where $\Pi_n(\Omega)$ denotes the space of polynomials of degree $n$ restricted on $\Omega$. 
This is where our fourth assertion comes in, which ensures the existence of fast decaying polynomials on
the domain $\Omega$. 

\medskip \noindent
{\it \bf Assertion 4}. {\it Let $\Omega$ be compact. For each $x \in \Omega$, there is a nonnegative 
polynomial $T_x$ of degree at most $n$ that satisfies 
\begin{enumerate}[   (1)]
\item $T_x(x) =1$, $T_x(y) \ge \delta > 0$ for $y \in B(x,\f 1 n)$ for some $\delta$ independent of $n$,
and, for each $\g > 1$,  
$$
     0 \le  T_x(y) \le c_\g (1+ n \sd(x,y))^{-\g}, \qquad y \in \Omega; 
$$
\item there is a polynomial $q_n$ such that $q_n(x) T_x(y)$ is a polynomial of degree at most $r n$,
for some positive integer $r$, in $x$-variable and $c_1 \le q_n(x) \le c_2$ for $x \in \Omega$ for some
positive numbers $c_1$ and $c_2$. 
\end{enumerate}  }

For conic domains, we will establish this assertion by explicitly construction. 

\subsubsection{Bernstein inequality for the spectral operator}

Our study of the best approximation by polynomials relies on the existence of a spectral operator
for orthogonal polynomials. 

\begin{defn}\label{def:LBoperator}
Let $\varpi$ be a weight function on $\Omega$. We denote by $\fD_\varpi$ the second order 
derivation operator that has orthogonal polynomials with respect to $\varpi$ as eigenfunctions; 
more precisely, 
\begin{equation}\label{eq:LBoperator}
  \fD_{\varpi} Y = - \mu(n) Y, \qquad \forall\, Y \in \CV_n(\Omega,\varpi),
\end{equation}
where $\mu$ is a nonnegative quadratic polynomial. 
\end{defn}

For the unit sphere, this is the Laplace-Beltrami operator. Its analogues exist on weighted sphere, ball
and simplexes (cf. \cite{DX}), as well as for conic domains \cite{X20a, X20b}. We will need the Bernstein 
inequality for the power of these operators. Since $\mu(k) \ge 0$, the operator $-\fD_\varpi$ is a non-negative
operator. A function $f \in L^p(\Omega;\varpi)$ belongs to the Sobolev space $W_p^r (\Omega; \varpi)$ 
if there is a function $g \in L^p(\Omega; \varpi)$, which we denote by $(-\fD_\varpi)^{\f r 2}f$, such that  
\begin{equation} \label{eq:fracDiff}
  \proj_n\left(\varpi; (-\fD_\varpi)^{\f r 2}f\right) = \mu(n)^{\f r 2} \proj_n(\varpi; f),
\end{equation}
where we assume that $f, g \in C(\Omega)$ when $p = \infty$. The fractional differential operator 
$(-\fD_\varpi)^{\f r 2}f$ is a linear operator on the space $W_p^r (\Omega; \varpi)$ defined by 
\eqref{eq:fracDiff}.

For $r > 0$, we denote by $L_{n}^{(r)}(\varpi; \cdot,\cdot)$ the kernel defined by 
\begin{equation} \label{eq:Ln_r}
   L_{n}^{(r)}(\varpi; x,y) = \sum_{n=0}^\infty \wh a\left( \frac{k}{n} \right) [\mu(k)]^{r/2} P_k(\varpi; x,y),
\end{equation}
which is the kernel $\fD_\varpi^{r/2} L_n(x,y)$ with $\fD_\varpi^{r/2}$ applying on $x$ variable. Our Bernstein 
inequality is proved under the following assumption on the decaying of this kernel. 

\medskip\noindent 
{\bf Assertion 5.} {\it For $r > 0$ and $\k > 0$, the kernel $ L_{n}^{(r)}(\varpi)$ satisfies, for $x,y \in \Omega$,
$$
  \left|  L_{n}^{(r)}(\varpi; x,y) \right| \le c_\k \frac{n^r}{\sqrt{\varpi(B(x,n^{-1}))}\sqrt{\varpi(B(y,n^{-1}))}}
     (1+n\sd(x,y))^{-\k}.
$$ }
\smallskip

For $r=0$, this reduces to the Assertion 1. We list it separately since Assertion 5 is not needed for
the localized polynomial frames. 

\subsection{Classical orthogonal polynomials}

In this subsection we collect a few results on classical orthogonal polynomials that we will need. 
There are three families, Jacobi polynomials, spherical harmonics and classical orthogonal polynomials
on the ball. Each of three families comes from an example of a localized homogeneous space, as 
explained in \cite{X21}, and has been extensively studied. We will recall only the basics. 

\subsection{Jacobi polynomials}
For $\a, \b > -1$, the Jacobi weight function is defined by 
$$
      w_{\a,\b}(t):=(1-t)^\a(1+t)^\b, \qquad -1 < x <1. 
$$
Its normalization constant $c'_{\a,\b}$, defined by $c'_{\a,\b}  \int_{-1}^1 w_{\a,\b} (x)dx = 1$, is given by
\begin{equation}\label{eq:c_ab}
 c'_{\a,\b} = \frac{1}{2^{\a+\b+1}} c_{\a,\b} \quad\hbox{with} \quad 
   c_{\a,\b} := \frac{\Gamma(\a+\b+2)}{\Gamma(\a+1)\Gamma(\b+1)}.
\end{equation}
The Jacobi polynomials $P_n^{(\a,\b)}$ satisfy the orthogonal relations
$$
c_{\a,\b}' \int_{-1}^1 P_n^{(\a,\b)}(x) P_m^{(\a,\b)}(x) w_{\a,\b}(x) \d x = h_n^{(\a,\b)} \delta_{m,n},
$$
where $h_n^{(\a,\b)}$ is the square of the $L^2$ norm that satisfies
$$
  h_n^{(\a,\b)} =  \frac{(\a+1)_n (\b+1)_n(\a+\b+n+1)}{n!(\a+\b+2)_n(\a+\b+2 n+1)}.
$$
  
Let $\wh a$ be an admissible cut-off function. We defined a polynomial $L_n^{(\a,\b)}$ by 
\begin{equation}\label{def.L}
L_n^{(\a,\b)}(t)=\sum_{j=0}^\infty \wh a \Big(\frac{j}{n}\Big)
       \frac{P_j^{(\a,\b)}(t) P_j^{(\a,\b)}(1)} {h_j^{(\a,\b)}}. 
\end{equation}
The estimate stated below (\cite{BD} and \cite[Theorem 2.6.7]{DaiX}) will be used several times. 

\begin{thm}
Let $\ell$ be a positive integer and let $\eta$ be a function that satisfy, $\eta\in C^{3\ell-1}(\RR)$, 
$\mathrm{supp}\, \eta \subset [0,2]$ and $\eta^{(j)} (0) = 0$ for $j = 0,1,2,\ldots, 3 \ell-2$. Then, 
for $\a \ge \b \ge -\f12$, $t \in [-1,1]$ and $n\in \NN$, 
\begin{equation} \label{eq:DLn(t,1)}
\left| \frac{d^m}{dt^m} L_n^{(\a,\b)}(t) \right|  \le c_{\ell,m,\a}\left\|\eta^{(3\ell-1)}\right\|_\infty 
    \frac{n^{2 \a + 2m+2}}{(1+n\sqrt{1-t})^{\ell}}, \quad m=0,1,2,\ldots . 
\end{equation}
\end{thm}

The Jacobi polynomials with equal parameters are the Gegenbauer polynomials $C_n^\l$, which 
are orthogonal with respect to $w_\l(x) = (1-x^2)^{\l-\f12}$, 
\begin{equation} \label{eq:GegenNorm}
  c_{\l} \int_{-1}^1 C_n^{\l}(x) C_m^{\l}(x) w_\l(x) \d x = h_n^{\l} \delta_{n,m}, 
\end{equation}
where the normalization constant $c_\l$ of $w_\l$ and the norm square $h_n^\l$ are given by
\begin{equation}\label{eq:c_l}
    c_{\l} = \frac{\Gamma(\l+1)}{\Gamma(\f12)\Gamma(\l+\f12)} \qquad \hbox{and} \qquad
      h_n^\l = \frac{\l}{n+\l}C_n^\l(1) = \frac{\l}{n+\l}\frac{(2\l)_n}{n!}.
\end{equation}

\subsubsection{Spherical harmonics}
Spherical harmonics are homogeneous polynomials that are orthogonal on the unit sphere. 
Let $\CH_n(\sph)$ be the space of spherical harmonics of degree $n$ in $d$ variables. Then  
$\dim \CH_n(\sph) = \binom{n+d-2}{n}+\binom{n+d-3}{n-1}$. 
For $n \in \NN_0$ let $\{Y_\ell^n: 1 \le \ell \le \dim \CH_n(\sph)\}$ be an 
orthonormal basis of $\CH_n(\sph)$; then 
$$
   \frac{1}{\o_d} \int_\sph Y_\ell^n (\xi) Y_{\ell'}^m (\xi)\d\s(\xi) = \delta_{n,m},
$$
where $\d \s$ is the surface measure of $\sph$ and $\o_d$ denotes the surface area 
$\o_d = 2 \pi^{\f{d}{2}}/\Gamma(\f{d}{2})$ of $\sph$. The spherical harmonics satisfy two
characteristic properties. In terms of the orthonormal basis $\{Y_\ell^n\}$, its reproducing
kernel $\sP_n(\xi,\eta)$ satisfies the addition formula
\begin{equation} \label{eq:additionF}
  \sP_n(\xi,\eta) = \sum_{\ell =1}^{\dim \CH_n(\sph)} Y_\ell^n (\xi) Y_\ell^n(\eta) = Z_n^{\f{d-2}{2}} (\la \xi,\eta\ra), 
     \quad \xi, \eta \in \sph,
\end{equation}
where $Z_n^\l$ is defined in terms of the Gegenbauer polynomial by 
\begin{equation} \label{eq:Zn}
    Z_n^\l(t) = \frac{n+\l}{\l} C_n^\l(t), \qquad \l = \frac{d-2}{2}. 
\end{equation}
Furthermore, spherical harmonics are the eigenfunctions of the Laplace-Beltrami operator $\Delta_0$, 
which is the restriction of the Laplace operator $\Delta$ on the unit sphere. More precisely (cf. \cite[(1.4.9)]{DaiX})
\begin{equation} \label{eq:sph-harmonics}
     \Delta_0 Y = -n(n+d-2) Y, \qquad Y \in \CH_n^d. 
\end{equation}

\subsubsection{Orthogonal polynomials on the unit ball}
The classical weight function on the unit ball $\BB^d$ of $\RR^d$ is defined by 
\begin{equation}\label{eq:weightB}
  W_\mu(x) = (1-\|x\|)^{\mu-\f12}, \qquad x\in \BB^d, \quad \mu > -\tfrac12.
\end{equation}
Its normalization constant is $b_\mu^\BB = \Gamma(\mu+\f{d+1}{2}) /(\pi^{\f d 2}\Gamma(\mu+\f12))$.
Let $\CV_n^d(\BB^d,W_\mu)$ be the space of orthogonal polynomials of degree $n$ with respect to 
$W_\mu$. An orthogonal basis of $\CV_n^d(\BB^d, W_\mu)$ can be given explicitly in terms of 
the Jacobi polynomials and spherical harmonics. For $ 0 \le m \le n/2$, let $\{Y_\ell^{n-2m}: 1 \le \ell \le
 \dim \CH_{n-2m}(\sph)\}$ be an orthonormal basis of $\CH_{n-2m}^d$. Define \cite[(5.2.4)]{DX}
\begin{equation}\label{eq:basisBd}
  P_{\ell, m}^n (x) = P_m^{(\mu-\f12, n-2m+\f{d-2}{2})} \left(2\|x\|^2-1\right) Y_{\ell}^{n-2m}(x).
\end{equation}
Then $\{P_{\ell,m}^n: 0 \le m \le n/2, 1 \le \ell \le \dim \CH_{n-2m}(\sph)\}$ is an orthogonal basis of 
$\CV_n^d(W_\mu)$.  Let $\Pb_n(W_\mu; \cdot, \cdot)$ denote the reproducing kernel of the 
space $\CV_n(\BB^d, W_\mu)$. This kernel satisfies an analogue of the addition formula \cite{X99}: 
for $x,y \in \BB^d$,
\begin{align} \label{eq:additionBall}
      \Pb_n(W_\mu;x,y) = c_{\mu-\f12} 
         \int_{-1}^1 Z_n^{\mu+\f{d-1}{2}} & \Big(\la x, y \ra + u \sqrt{1-\|x\|^2}\sqrt{1-\|y\|^2} \Big) \\
           &  \times (1-u^2)^{\mu-1}\d u, \notag
\end{align}
where $\mu > 0$ and it holds for $\mu =0$ under the limit
\begin{equation}\label{eq:limitInt}
   \lim_{\mu \to 0+}  c_{\mu-\f12} \int_{-1}^1 f(t) (1-t^2)^{\mu-1} \d u = \frac{f(1) + f(-1)}{2}. 
\end{equation} 
The orthogonal polynomials with respect to $W_\mu$ on the unit ball are eigenfunctions of a second 
order differential operator \cite[(5.2.3)]{DX}: for all $u \in \CV_n(\BB^d, W_\mu)$,
\begin{equation}\label{eq:diffBall}
  \left( \Delta  - \la x,\nabla \ra^2  - (2\mu+d-1) \la x ,\nabla \ra \right)u = - n(n+2\mu+ d-1) u
\end{equation}
  
\section{Homogeneous space on double conic and hyperbolic surfaces} 
\setcounter{equation}{0}

In this chapter we work in the setting of homogeneous space on the surface  
$$
  \XX_0^{d+1} =  \left \{(x,t): \|x\|^2 = t^2 - \varrho^2, \, x \in \RR^d, \, \varrho \le |t| \le \sqrt{\varrho^2 +1}\right\}, 
$$
which is a hyperbolic surface of two sheets when $\varrho > 0$ and a double conic  surface when $\varrho = 0$. 
We shall verify that this domain admits a localized homogeneous space for a family of weight functions 
$\sw_{\b,\g}^\varrho$, which are related to the Gegenbauer weight functions and are even in $t$ variable. 

For $\varrho =0$, the upper part of $\XX_0^{d+1}$ with $t \ge 0$ is exactly the conic surface $\VV_0^{d+1}$.
The analysis on $\XX_0^{d+1}$, however, is of a different character. As a starter, the distance on 
$\XX_0^{d+1}$ is comparable to the Euclidean distance, in contrast to the distance on $\VV_0^{d+1}$. 
Moreover, the space of orthogonal polynomials are divided naturally into two orthogonal subspaces, consisting
of polynomials even in $t$ variable or odd in $t$ variable. We consider only orthogonal polynomials even in 
$t$ variable, since they alone satisfy the two characteristic properties: addition formula and the differential 
operator having orthogonal polynomials as eigenfunctions. As a consequence, we consider approximation 
and localized frames for functions that are even in $t$ variable. Many estimates and computations are 
easier to handle comparing with the conic surface because of simpler distance function and an addition
formula that is less cumbersome. The structure of the chapter follows that of the
previous two chapters, with contents arranging in the same order under similar section names.

\subsection{Distance on double conic and hyperbolic surfaces}
Whenever it is necessary to emphasis the dependence on $\varrho$, we write ${}_\varrho \XX_0^{d+1}$
instead of $\XX_0^{d+1}$. In most places, however, we use $\XX_0^{d+1}$ to avoid excessive notations. 

The surface $\XX_0^{d+1}$ can be decomposed as an upper part and a lower part, 
$$
    \XX_0^{d+1}  = \XX_{0,+}^{d+1} \cup \XX_{0,-}^{d+1} =  \{(x,t) \in \XX_0^{d+1}: t \ge 0\} \cup
      \{(x,t) \in \XX_0^{d+1}: t \le 0\}. 
$$
For $\varrho =0$, the upper part $\XX_{0,+}^{d+1}$ is exactly the conic surface $\VV_0^{d+1}$. 

As $\XX_0^{d+1}$ is a bounded domain, its distance function should take into account of the boundary 
at the two ends, but it does not regard the apex as a boundary point because of symmetry. We first define 
a distance on the double conic surface; that is, when $\varrho =0$.

\begin{defn}
Let $\varrho = 0$. For $(x,t), (y,s) \in \XX_0^{d+1}$, define 
$$
   \sd_{\XX_0}( (x,t),(y,s)) = \arccos \left(\la x,y\ra + \sqrt{1-t^2} \sqrt{1-s^2} \right).
$$
\end{defn}

To see that $\sd_{\XX_0}(\cdot,\cdot)$ is indeed a distance on the double conic surface $\XX_0^{d+1}$, 
we let $X= \big(x , \sqrt{1-t^2}\big )$ and $Y= \big (y, \sqrt{1-s^2}\big )$, so that $X, Y \in \SS^{d}$ if 
$x, y \in \XX_0^{d+1}$. Then 
$$
       \sd_{\XX_0}( (x,t),(y,s)) = \sd_{\SS^d} (X,Y), \quad (x,t), (y,s) \in \XX_0^{d+1},
$$
from which it follows readily that $\sd_{\XX_0}$ is indeed a distance on $\XX_0^{d+1}$.

This distance is quite different from the distance $\sd_{\VV_0}(\cdot,\cdot)$ of $\VV_0^{d+1}$ 
defined by 
$$
  \sd_{\VV_0}( (x,t),(y,s)) = \arccos \left(\sqrt{\frac{t s + \la x,y\ra}{2}} + \sqrt{1-t} \sqrt{1-s} \right), 
   \quad (x,t), (y,s) \in \VV_0^{d+1}
$$
in \cite{X21} and it resembles the distance $\sd_{\BB}(\cdot,\cdot)$ of the unit ball which is of the same 
form but with $t$ and $s$ replaced by $\|x\|$ and $\|y\|$. The distance $\sd_{\XX_0}$ is related to the 
distance on the sphere, defined by  
\begin{equation}\label{eq:distSph}
\sd_{\SS}(\xi,\eta) = \arccos \la \xi,\eta\ra, \qquad \xi,\eta\in \sph,
\end{equation}
and the distance on the interval $[-1,1]$, defined by 
\begin{equation}\label{eq:dist[-1,1]}  
\sd_{[-1,1]}(t,s) = \arccos \left(t s + \sqrt{1-t^2} \sqrt{1-s^2}\right), \qquad t, s \in [-1,1].
\end{equation}

\begin{prop}\label{prop:cos-distX0}
Let $\varrho  =0$ and $d \ge 2$. For $(x,t), (y,s) \in \XX_0^{d+1}$, write $x =t\xi$ and $y = s \eta$ with 
$\xi,\eta\in \sph$, it holds 
\begin{equation}\label{eq:d=d+dX0}
   1- \cos \sd_{\XX_0} ((x,t), (y,s)) =1-\cos \sd_{[-1,1]}(t,s) + t s \left(1-\cos \sd_{\SS}(\xi,\eta) \right).
\end{equation}
In particular, 
\begin{equation} \label{eq:d2=d2+d2X0}
  c_1 \sd_{\XX_0} ((x,t), (y,s))  \le \sd_{[-1,1]}(t,s) + \sqrt{t s} \, \sd_{\SS}(\xi,\eta)  \le c_2 \sd_{\XX_0} ((x,t), (y,s)).
\end{equation}
\end{prop}
 
\begin{proof} 
Using $\la x ,y \ra \mathrm{sign}(t s) = t s \la \xi,\eta\ra$, we obtain 
\begin{align*}
  1- \cos  \sd_{\XX_0} ((x,t), (y,s)) = 1 - t s \la \xi,\eta\ra - \sqrt{1-t^2} \sqrt{1-s^2}. 
\end{align*}
Hence, \eqref{eq:d=d+dX0} follows immediately from \eqref{eq:dist[-1,1]} and \eqref{eq:distSph}. Moreover,
using \eqref{eq:d=d+dX0}, the estimate \eqref{eq:d2=d2+d2X0} follows from $1-\cos \t = 2 \sin^2\frac{\t}{2}$, 
$\frac{1}{\pi} \t \le \sin \f{\t}{2} \le \f{\t}{2}$ for $0 \le \t \le \pi$, and $(a+b)^2/2 \le a^2+b^2 \le (a+b)^2$ for $a,b \ge 0$. 
\end{proof}

In particular, the distance on the linear segment $l_\xi= \{(t \xi, t): -1 \le t \le 1\}$ for any fixed $\xi \in \sph$
on the double conic  surface becomes 
$$
      \sd_{\XX_0}\big( (t \xi, t), (s \xi, s)\big) = \sd_{[-1,1]}(t,s). 
$$
We note that the line $l_\xi$ passes through the origin when it passes from the upper conic surface
to the lower conic surface; in fact, the line passes through $(\xi, 1)$ on the rim of the upper conic surface 
to $(-\xi,-1)$ on the opposite rim of the lower conic surface. 

\begin{rem}\label{rem:distX0}
In \cite[Remark 4.1]{X21}, it is pointed out that the distance $\sd_{\VV_0}(\cdot,\cdot)$ is incompatible 
with the Euclidean distance at around $t=0$ since, for $t =s$, 
$$
   \sd_{\VV_0}((x,t),(y,t)) \sim \sqrt{t} \d_{\SS}(\xi,\eta)\quad \hbox{and} \quad  \|(x,t) - (y,t)\| \sim t  \d_{\SS}(\xi,\eta).
$$
In contrast, the distance $\sd_{\XX_0}(\cdot,\cdot)$ on the double conic surface satisfies, 
$$
 \sd_{\XX_0}((x,t),(y,t)) \sim |t| \d_{\SS}(\xi,\eta) \sim  \|(x,t) - (y,t)\|. 
$$
\end{rem}

We will also need the following lemma in the proof of our estimates. 

\begin{lem} \label{lem:|s-t|X0}
Let $\varrho = 0$ and $d \ge 2$. If $(x,t), (y,s)$ both in $\XX_{0,+}^{d+1}$ or both in $\XX_{0,-}^{d+1}$, then
$$ 
    \big| t - s \big| \le \sd_{\XX_0} ((x,t), (y,s)) \quad\hbox{and}\quad
      \big| \sqrt{1-t^2} - \sqrt{1-s^2} \big| \le \sd_{\XX_0} ((x,t), (y,s)).
$$
\end{lem} 

\begin{proof}
It suffices to consider the case $t, s \ge 0$. Let $t = \cos \t$ and $s = \cos \phi$ with $0 \le \t, \phi \le \pi/2$. 
It follows readily that $\cos \sd_{\XX_0} ((x,t), (y,s)) \le t s + \sqrt{1-t^2} \sqrt{1-s^2} = \cos (\t-\phi)$, which is 
equivalently to $|\t-\phi| \le \sd_{\XX_0}((x,t), (y,s))$. Hence, the stated inequalities follows from 
$$
  \big| t - s \big| = |\cos \t - \cos \phi| = 2  \sin \frac{\t+\phi}{2}\sin \frac{|\t-\phi|}{2} \le |\t-\phi|, 
$$
and, similarly, 
$$
 \left| \sqrt{1-t^2} - \sqrt{1-s^2} \right| =  |\sin \t - \sin \phi|  = 2 \cos \frac{\t+\phi}{2}\sin \frac{|\t-\phi|}{2} \le |\t-\phi|
$$
for all $t, s\in [-1,1]$. This completes the proof. 
\end{proof}

For $\varrho > 0$, the hyperbolic surface $\XX_0^{d+1} = \XX_{0,+}^{d+1} \cup \XX_{0,-}^{d+1}$ consists of 
two disjoint parts. For all practical purpose, it is sufficient to consider the distance between points that lie in 
the same part. For $(x,t)$ and $(y,t)$ both in $\XX_{0,+}^{d+1}$ or both in $\XX_{0,-}^{d+1}$, we define 
\begin{align}\label{eq:distXX0rho}
 \sd_{\XX_0}^\varrho( (x,t), (y,s)) & = \arccos \left (\la x,y\ra + \sqrt{1+\varrho^2-t^2}\sqrt{1+\varrho^2-s^2} \right) \\
         & = \sd_{\XX_0} \left( \Big(x, \sqrt{t^2-\varrho^2}\Big), \Big(y,\sqrt{s^2 - \varrho^2}\Big)\right). \notag
\end{align}
It is easy to see that this is a distance function and, evidently, we can derive its properties as we did for 
the distance on the double conic  surface. 

\subsection{A family of doubling weights} 
For $d \ge 2$, $\b > -\f12$ and $\g > -\f12$, let $\sw_{\b,\g}^\rho$ be the weight function defined on the 
hyperbolic surface $\XX_0^{d+1}$ by 
\begin{equation}\label{eq:sf6weight}
   \sw_{\b,\g}^\varrho(t) = 
      |t| (t^2-\varrho^2)^{\b-\f12}(\varrho^2+1 - t^2)^{\g-\f12},  \qquad  \varrho \ge 0,
\end{equation}
for $\varrho \le |t| \le \sqrt{\varrho^2 +1}$. When $\varrho = 0$, or on the double conic surface, it becomes 
$$
    \sw_{\b,\g}^0(t)= |t|^{2\b}(1-t^2)^{\g-\f12}, 
$$
which is integrable if $\b > -\frac{d+1}{2}$ on $\XX_0^{d+1}$. When $\b =0$, $\sw_{0,\g}^0$ is the classical
weight function for the Gegenbauer polynomial $C_n^\g$. Evidently, the two cases are related, 
$$
   \sw_{\b,\g}^\varrho(t) = |t| \sw_{\b-\f12,\g}^0\left(\sqrt{t^2-\varrho^2}\right). 
$$
Let $\d \s_\varrho$ denote the surface measure on
${}_\varrho \XX_0^{d+1}$. Then $\d \s_\varrho(x,t) =  \d\omega_{\sqrt{t^2-\varrho^2}} (x) \d t$, where 
$\d \omega_{r}$ denotes the surface measure on the sphere $\SS^{d-1}_{r}$ of radius $r$. It follows then
\begin{align} \label{eq:intw(t)X0}
 \int_{{}_\varrho \XX_0^{d+1}} f (x,t)|t| \d \s_\varrho (x,t)& =  
     \int_{\varrho \le |t|\le \varrho +1} |t| \int_{\|x\| = \sqrt{t^2-\varrho^2}} f(x,t) \d \omega_{\sqrt{t^2-\varrho^2}} (x) \d t\\
      &    = \int_{|s|\le 1} \int_{\|x\| = |s|} f (x, \sqrt{s^2+\varrho^2})|s| \d \omega_{|s|}(x) \d s\notag  \\
      &     =  \int_{{}_0\XX_0^{d+1}} f (x, \sqrt{s^2+\varrho^2})|s| \d \s(x,s). \notag
\end{align}
Using this relation, it is easy to verify that the normalization constant $\bs_{\b,\g}^\varrho$ of the weight function
$\sw_{\b,\g}^\varrho$ is given by
$$
\bs_{\b,\g}^\varrho = \bs_{\b,\g} = \frac{\Gamma(\b+\g+\f{d+1}{2})}{\s_d\Gamma(\b+\f d 2)\Gamma(\g+\f12)}.
$$ 

For $r > 0$, $(x,t)$ on the hyperbolic surface ${}_\varrho \XX_0^{d+1}$ and $(x,t) \ne (0,0)$, we denote the ball centered 
at $(x,t)$ with radius $r$ by 
$$
   \sc_\varrho((x,t), r): = \left\{ (y,s) \in {}_\varrho\XX_0^{d+1}: \sd_{\XX_0}^\varrho \big((x,t),(y,s)\big)\le r \right\}.
$$   

\begin{lem}\label{lem:cap-rhoX0} 
For $\varrho \ge 0$ and $(x,t) \in \XX_0^{d+1}$,
$$
 \sw_{\b,\g}^\varrho \big(\sc_\varrho((x,t), r)\big) 
  = \sw_{\b,\g}^0 \left(\sc_0\left(\big(x, \sqrt{t^2-\varrho^2}\big), r\right)\right). 
$$
\end{lem}

\begin{proof}
By the definition of $\sw (E)$ and \eqref{eq:distXX0rho}, 
\begin{align*}
  \sw_{\b,\g}^\varrho \big( \sc_\varrho((x,t), r)\big) \,& =  \bs_{\b,\g}\int_{\sc_\varrho((x,t), r)} \sw_{\b,\g}^\varrho(s) \d \s_\varrho(y,s) \\
   & =  \bs_{\b,\g} \int_{\sd_{\XX_0} \big((x,\sqrt{t^2-\|x\|^2}), (y,\sqrt{s^2-\|y\|^2})\big) \le r} \sw_{\b,\g}^\varrho(s) \d \s_\varrho (y,s)\\
      & = \bs_{\b,\g} \int_{\sd_{\XX_0} \big((x,\sqrt{t^2-\|x\|^2}), (y,s)\big) \le r} \sw^0_{\b,\g} (s) \d \s_0 (y,s)\\
      &  = \sw_{\b,\g}^0 \left(\sc_0\left(\big(x, \sqrt{t^2-\varrho^2}\big), r\right)\right),
\end{align*}
where we have used \eqref{eq:intw(t)X0} in the second to last step. 
\end{proof}
 
\begin{prop}\label{prop:capX0}
Let $r > 0$ and $(x,t) \in \XX_0^{d+1}$. For $\b > - \f{d+1}2$ and $\g > - \f12$, 
\begin{align}\label{eq:capX0}
 \sw_{\b,\g}^0 \big(\sc_0((x,t), r)\big)\, & = \bs_{\b,\g}^0 \int_{ \sc_0((x,t), r)} \sw^0_{\b,\g}(s) \d \s(y,s)  \\
     & \sim r^d \big(t ^2+ r^2\big)^{\b} \big(1 -t^2+ r^2\big)^{\g}.  \notag
\end{align}
In particular, $ \sw_{\b,\g}^0 $ is a doubling weight on the double conic surface and the doubling index 
is given by $\a(\sw_{\b,\g}^0) = d + 2 \max\{0, \b\} + 2 \max\{0,\g\}$. 
\end{prop}

\begin{proof} 
Without loss of generality, we assume $r$ is bounded by a small positive number, say $r \le \f{\pi}{12}$ and
$0\le t \le 1$. Since $\sw_{\b,\g}^0$ is even in $t$, we only need to work on $\XX_{0,+}^{d+1} = \VV_0^{d+1}$. 
Let $x = t \xi$ and $y = s \eta$ for $\xi,\eta \in \sph$.  
By the inequality \eqref{eq:d=d+dX0}, from $\sd_{\XX_0}((x,t),(y,s)) \le r$ we obtain $\sd_{[-1,1]}(t, s) \le r$; moreover, denote  $\tau_r(t,s) = (\cos r - \sqrt{1-t^2}\sqrt{1-s^2} )/ (ts)$ and $\t_r(t,s) = \arccos \tau_r(t,s)$, we also have $\sd_{\SS} (\xi,\eta) \le \tfrac12 \arccos \tau_r(t,s) = \tfrac12 \t_r(t,s)$. Then it is easy to see, since $\d \s(y,s) = 
s^{d-1} \d \s_\SS(\eta) \d s$, that 
\begin{align*}
   \sw_{\b,\g}^0 \big(\sc((x,t), r)\big)&  = \int_{\sd_{[-1,1]}(t, s)\le r} s^{d-1} 
        \int_{\sd_{\SS}(\xi,\eta) \le \tfrac12 \t_r(t,s)} \sw_{\b,\g}^0(s\eta, s) \d \s_{\SS}(\eta)\d s.
\end{align*}
By symmetry, we can choose $\xi =  (1,0,\ldots,0)$ and use the identity (cf. \cite[(A.5.1)]{DaiX}) 
\begin{equation} \label{eq:intSS}
  \int_{\sph} g(\la \xi,\eta\ra) d\s(\eta) 
      = \o_{d-1} \int_0^\pi g (\cos \t) (\sin\t)^{d-2} \d \t
\end{equation} 
with $\o_{d-1}$ being the surface are of $\SS^{d-2}$, to obtain that 
\begin{align*}
   \sw_{\b,\g}^0\big(\sc((x,t), r)\big) = \o_{d-1} \int_{\sd_{[0,1]}(t, s)\le r}  s^{d-1} \sw_{\b,\g}^0(s) \int_0^{\tfrac12 \t_r(t,s)}
            (\sin \t)^{d-2} \d \t \d s. 
\end{align*}
Since $\t \sim \sin \t \sim \sqrt{1-\cos \t}$, it follows then 
\begin{equation} \label{eq:sw(cap)}
   \sw_{\b,\g}^0 \big(\sc_0((x,t), r)\big) \sim \int_{\sd_{[-1,1]}(t, s)\le r}  s^{d-1} \sw^0_{\b,\g}(s) 
          \big(1-\tau_r(t,s) \big)^{\f{d-1}2} \d s. 
\end{equation} 
We now need to consider three cases. If $3 r < t < 1- 3 r$, then we can use Lemma \ref{lem:|s-t|X0} to conclude 
that $s^2 \sim t^2+ r^2$ and $1-s^2 \sim 1-t^2+r^2$, so that 
\begin{align*}
 \sw_{\b,\g}^0 \big(\sc_0((x,t), r)\big) & \, \sim (t^2+ r^2)^\b (1-t^2+ r^2)^{\g} \\
     & \times \int_{\sd_{[-1,1]}(t, s)\le r} \big(\cos (\sd_{[-1,1]} (t,s)) - \cos r \big)^{\f{d-1}2} 
    \frac{\d s}{\sqrt{1-s^2}}. 
\end{align*}
Setting $t = \cos \t$ and $s = \cos \phi$ so that $\sd_{[-1,1]} (t,s) = |\t-\phi|$, then the last integral is easily 
seen to be
$$
  \int_{|\t-\phi|\le r} \big(\cos(\t-\phi) - \cos r \big)^{\f{d-1}2} \d \phi
    =  c \int_{|\zeta|\le r} \big(\sin \tfrac{\zeta-r}{2} \sin \tfrac{\zeta+r}{2} \big)^{\f{d-1}2} \d \zeta  \sim r^d.
$$
This completes the proof of the first case. If $t \le 3r$, then $|t-s| \le d_{[-1,1]} (t,s) \le r$ so
that $s \le 4 r$. Evidently $1-s^2 \sim 1-t^2 \sim 1$ in this case. Furthermore, let $t = \sin \t$ and $s = \sin \phi$;
then $ |\t-\phi| = d_{[-1,1]}(t,s)\le r$, which is easily seen to be equivalent to $|\tau_r(t,s)| \le 1$. Hence, using 
$1-\tau_r(t,s) \le 2$ and $s \le 4 r$, we obtain by \eqref{eq:sw(cap)} and $\sin \phi \sim \phi$,
\begin{align*}
  \sw_{\b,\g}\big(\sc((x,t), r)\big) \le c  \int_{\sd_{[-1,1]}(t, s)\le r}   s^{d+2\b -1} \d s \sim r^{2\b+ d},
\end{align*}
which proves the upper bound in \eqref{eq:capX0}. For the lower bound, we consider a subset of $\sc((x,t),r)$ 
with $\d_{[-1,1]}(t,s) \le r/2$. Using the upper bound of $s$ and $t$, we then deduce 
$$
  1-\tau_r(t,s)= \frac{\cos d_{[-1,1]}(t,s) - \cos r}{t s} \ge \frac{\cos \tfrac{r}{2} - \cos r}{ 12 r^2} \ge \frac1{8\pi^2},
$$
where in the last step we have used $\sin \t \ge \frac{2}{\pi}\t$, which shows that 
\begin{align*}
 \sw_{\b,\g}\big(\sc((x,t), r)\big)\,& \ge c \int_{\sd_{[-1,1]}(t, s) \le r}  s^{d+ 2 \b -1} \d s \sim r^{2\b+ d}. 
\end{align*}
Finally, if $t \ge 1- 3 r$, then we have $s \sim t  \sim 1$ for $(y,s) \in \sc((x,t),r)$. Since
$\sd_{[-1,1]}(t,s) = \sd_{[-1,1]}(\sqrt{1-t^2},\sqrt{1-s^2})$, changing variable $s\mapsto \sqrt{1-s^2}$ in  
\eqref{eq:sw(cap)}, we see that this case can be reduced to that of the second case.
This completes the proof of \eqref{eq:capX0}.
\end{proof}

For $\varrho > 0$ and $\b, \g > -\f12$, the weight function $\sw_{\b,\g}^\varrho$ is a doubling weight
on the hyperbolic surface ${}_\varrho \XX^{d+1}$ by Lemma \ref{lem:cap-rhoX0}.

\begin{cor}
For $\varrho \ge 0$, $d\ge 2$, and $\b, \g > -\f12$, the space 
$(\XX_0^{d+1}, \sw_{\b,\g}^\varrho, \sd_{\XX_0}^\varrho)$ is a homogeneous space. If $\varrho =0$,
the restriction on $\b$ can be relaxed to $\b > -\frac{d+1}{2}$.
\end{cor}

When $\varrho =0$, $\b = 0$ and $\g = \f12$, the relation \eqref{eq:capX0} is for the Lebesgue measure 
$\sigma_0$ on the double conic surface; in particular, if $\b = 0$ and $\g = 0$, then $\sw_{0,0,}^0(t) = (1-t^2)^{-\f12}$ 
is the Chebyshev weight and $\sw_{0,0}^0 \big(\sc_0((x,t), r)\big)  \sim r^d$. The apex point $t = 0$ does not appear 
as a boundary point of $\XX_0^{d+1}$. 

For convenience, we will introduce the function $\sw_{\b,\g}(n;t)$ defined by 
\begin{equation}\label{eq:w_bgHyp}
   \sw_{\b,\g}^\varrho(n; t) = n^{-d}  \sw_{\b,\g}^\varrho \big(\cb_\varrho((x,t),n^{-1}) \big) =
       \big(t ^2-\varrho^2 + n^{-2}\big)^{\b} \big(1 -t^2 + \varrho^2+ n^{-2}\big)^{\g} 
\end{equation}
on the hyperbolic surface $\XX_0^{d+1}$ and use it in latter sections.

\subsection{Orthogonal polynomials on hyperbolic surfaces}
For $\g > -\f12$, we define the inner product on the hyperbolic surface by
$$
  \la f,g\ra_{\sw_{\b,\g}} = \bs_{\b,\g} \int_{\XX_0^{d+1}} f(x,t) g(x,t) \sw_{\b,\g}^\varrho(t) \d \s_\varrho(x,t), 
$$ 
where $\b > -\f12$ if $\varrho>0$ and $\b > -\frac{d+1}{2}$ if $\varrho =0$. The orthogonal polynomials 
with respect to this inner product are studied in \cite{X20b}. 
Let $\CV_n(\XX_0^{d+1}, \sw_{\b,\g}^\varrho)$ be the space of these polynomials with respect to
this inner product, which has the same dimension as the space of spherical harmonics $\CH_n^{d+1}$. 
Because the weight function $\sw_{\b,\g}^\varrho$ is even in $t$,  this space can be factored as 
$$
  \CV_n\left(\XX_0^{d+1}, \sw_{\b,\g}^\varrho\right) =  \CV_n^E\big(\XX_0^{d+1}, \sw_{\b,\g}^\varrho\big) \bigoplus 
      \CV_n^O\big(\XX_0^{d+1}, \sw_{\b,\g}^\varrho\big),
$$
where the subspace $\CV_n^E(\XX_0^{d+1}, \sw_{\b,\g}^\varrho)$ consists of orthogonal polynomials that 
are even in $t$ variable, and the subspace $\CV_n^O(\XX_0^{d+1}, \sw_{\b,\g}^\varrho)$ consists of 
orthogonal polynomials that are odd in $t$ variable. Moreover, 
\begin{equation*}
  \dim \CV_n^E\big(\XX_0^{d+1},\sw_{\b,\g}^\varrho\big) =\binom{n+d-1}{n}, \quad \quad  
    \dim \CV_n^O\big(\XX_0^{d+1},\sw_{\b,\g}^\varrho\big) =\binom{n+d-2}{n-1}. 
\end{equation*}

It turns out that an orthogonal basis can be given in terms of the spherical harmonics and the Jacobi 
polynomials for the subspace $\CV_n^E(\XX_0^{d+1},\sw_{\b,\g}^\varrho)$ for all $\varrho \ge 0$, but
for the subspace $\CV_n^O(\XX_0^{d+1},\sw_{\b,\g}^\varrho)$ only when $\varrho =0$. For example,
let $\{Y_\ell^m: 1 \le \ell \le \dim \CH_m^d\}$ denote an orthonormal basis of $\CH_m^d$.  Then the 
polynomials 
\begin{equation}\label{eq:sfOPhypG}
  \sC_{m,\ell}^n (x,t) = P_k^{(\g-\f12,n-2k+\b+ \f{d-2}{2})}\big(2t^2-2\varrho^2-1\big) Y_\ell^{n-2k}(x),
\end{equation}
where $1 \le \ell \le \dim \CH_{n-2k}^d$ and $0 \le k\le n/2$, form an orthogonal basis of  
$\CV_n^E(\XX_{0}^{d+1}, \sw_{\b,\g}^\varrho)$. Since we will not work directly with explicit bases,
we refer to \cite{X20b} for further information, where these polynomials are called, when $\b =0$, the
Gegenbauer polynomials on the hyperbolic surface or on the double conic surface when $\varrho=0$.

Let $\sP_n^E(\sw_{\b,\g}^\varrho; \cdot,\cdot)$ be the reproducing kernel of $\CV_n^E(\XX_0^{d+1},
\sw_{\b,\g}^\varrho)$, which can be written in terms of the basis \eqref{eq:sfOPhypG} as 
$$
\sP_n^E\left(\sw_{\b,\g}^\varrho; (x,t),(y,s)\right) = \sum_{m=0}^n \sum_{\ell =1}^{\dim  \CH_{n-2m}^d}
   \frac{ \sC_{m,\ell}^n (x,t) \sC_{m,\ell}^n (y,s)}{\la  \sC_{m,\ell}^n,  \sC_{m,\ell}^n\ra_{\sw_{\b,\g}}}.
$$
Let $\proj_n^E\big(\sw_{\b,\g}^\varrho\big): L^2\big(\XX_{0}^{d+1}, \sw_{\b,\g}^\varrho\big) \mapsto
\CV_n^E\big(\XX_{0}^{d+1}, \sw_{\b,\g}^\varrho\big)$ be the orthogonal projection operator. Then it
can be written as 
$$
    \proj_n^E(\sw_{\b,\g}^\varrho ;f)=  \bs_{\b,\g} \int_{\XX_0^{d+1}} 
       f(y) \sP_n^E(\sw_{\b,\g}^\varrho; \cdot, (y,s))\sw_{\b,\g}^\varrho (s) \d y \d s. 
$$
If $f$ is a function that is even in the variable $t$ on $\XX_0^{d+1}$, then its orthogonal projection on 
$\CV_n^O(\XX_{0}^{d+1}, \sw_{\b,\g}^\varrho)$ becomes zero, so that its Fourier orthogonal expansion
is given by 
\begin{equation} \label{eq:FourierX0}
  f = \sum_{n=0}^\infty \proj_n^E(\sw_{\b,\g}^\varrho; f).
\end{equation}
Hence, the kernel $\sP_n^E(\sw_{\b,\g}^\varrho; \cdot,\cdot)$ is meaningful for studying the Fourier 
orthogonal expansions on the hyperboloid. 

If $\varrho = 0$, then the upper part $\XX_{0,+}^{d+1} = \VV_0^{d+1}$ is the upper conic surface. The function 
$f(x,t)$ that is even in $t$ variable can be regarded as defined on $\VV_{0}^{d+1}$ or as the even extension
in $t$ variable of a function defined on the upper conic surface. Consequently, the Fourier expansion 
\eqref{eq:FourierX0} works for the function $f$ defined on the upper conic surface $\XX_{0,+}^{d+1}$ when 
$\varrho = 0$. The latter, however, is different from the Fourier expansions in the Jacobi polynomials on the 
conic surface $\VV_0^{d+1}$ discussed in \cite{X21}. 

The case $\b = 0$ is the most interesting since the orthogonal polynomials for $\sw_{0,\g}$ enjoy two 
characteristic properties. The first one is the spectral operator that has orthogonal polynomials as 
eigenfunctions. 

\begin{thm}\label{thm:sfHypGdiff}
Let  $\varrho > 0$ and $\g > -\f12$. Then for $x = t \xi$, $\xi \in \sph$, define the differential operator 
\begin{align*} 
 \Delta_{0,\g}^\varrho = \,& (1+\varrho^2-t^2) \left(1- \frac{\varrho^2}{t^2} \right) \partial_t^2 \\
       + & \left ( (1+\varrho^2-t^2) \frac{\varrho^2}{t^2} - (2 \g+d) (t^2-\varrho^2) \right)\frac{1}{t} \partial_t  
       + \frac{d-1}{t} \partial_t + \frac{1}{t^2- \varrho^2} \Delta_0^{(\xi)}. 
\end{align*}
Then the polynomials in $\CV_n^E(\VV_0^{d+1}, w_{0,\g}^\varrho)$ are eigenfunctions of 
$ \Delta_{0,\g}^\varrho$, 
\begin{align} \label{eq:sfHypGdiff}
   \Delta_{0,\g}^\varrho u = - n(n+2\g+d-1)u, \quad \forall u \in \CV_n^E(\VV_0^{d+1}, w_{0,\g}^\varrho).
\end{align}
\end{thm}

The second one is the addition formula for the reproducing kernel $\sP_n^E\big(\sw_{\b,\g}^\varrho;\cdot,\cdot\big)$, 
which is of the simplest form when $\b = 0$. 

\begin{thm}\label{thm:PEOkernelX0}
Let $d \ge 2$ and $\varrho \ge 0$. Then
\begin{enumerate}[ \quad\rm (a)]
\item For $\b, \g > -\f12$, 
\begin{align}\label{eq:sfPEhyp}
   \sP_n^E \big(\sw_{\b,\g}^\varrho; (x,t), (y,s)\big) 
    = \sP_n^E \left(\sw_{\b,\g}^0; \Big(x,\sqrt{t^2-\varrho^2}\Big),  \Big(y,\sqrt{s^2-\varrho^2} \Big)\right). 
\end{align}
\item For  $\varrho =0$, $\b =0$ and $\g \ge 0$,
\begin{align} \label{eq:sfPEadd0Hyp}
\sP_n^E \big (\sw_{0,\g}^0; (x,t),(y,s) \big) &= c_{\g}  \int_{-1}^1 Z_n^{\g+\frac{d-1}{2}}\big( \zeta (x,t,y,s;v) \big)
    (1-v^2)^{\g-1} \d v,
\end{align}
where $c_\g = c_{\g-1,\g-1}$ and 
$$
 \zeta (x,t,y,s;v) = \la x,y\ra \mathrm{sgin}(t s) + v \sqrt{1-t^2}\sqrt{1-s^2},
$$ 
and the case $\g=0$ holds under the limit \eqref{eq:limitInt}. 
\end{enumerate}
\end{thm}

The closed form formula \eqref{eq:sfPEadd0Hyp} is essential for studying highly localized kernels. 
 
\subsection{Highly localized kernels}
Let $\wh a$ be a cut-off function. For $(x,t)$, $(y,s) \in \XX_0^{d+1}$, define the kernel 
$\sL_n^E(\sw_{0,\g}^\varrho)$ by
$$
   \sL_n^E(\sw_{0,\g}^\varrho; (x,t),(y,s)) = \sum_{j=0}^\infty \wh a\left( \frac{j}{n} \right) 
      \sP_j^E\big(\sw_{0,\g}^\varrho; (x,t), (y,s)\big). 
$$
The kernel uses only orthogonal polynomials that are even in $t$ and $s$ variable, so that it is even in both
$t$ and $s$ variables. We show that this kernel is highly localized when $(x,t)$ and $(y,s)$ are either both
in $\XX_{0,+}^{d+1}$ or both in $\XX_{0,-}^{d+1}$. For $\g \ge 0$, recall by \eqref{eq:w_bgHyp}, 
$$
     \sw_{0,\g}^\varrho (n; t) :=  \big(1+\varrho^2 - t^2 +n^{-2}\big)^{\g}.
$$

\begin{thm} \label{thm:kernelX0}
Let $d\ge 2$ and $\g \ge 0$. Let $\wh a$ be an admissible cutoff function. Then,
for any $\k > 0$, either $(x,t), (y,s)$ both in $\XX_{0,+}^{d+1}$ or both in $\XX_{0,-}^{d+1}$,
\begin{equation*}
\left |\sL_n^E (\sw_{0,\g}^\varrho; (x,t), (y,s))\right|
\le \frac{c_\k n^d}{\sqrt{ \sw_{0,\g}^\varrho (n; t) }\sqrt{ \sw_{0,\g}^\varrho (n; s) }}
    \big(1 + n \sd_{\XX_0}^\varrho( (x,t), (y,s)) \big)^{-\k},
\end{equation*}
where we assume $t$ and $s$ have the same sign when $\varrho > 0$.
\end{thm}

\begin{proof}
By \eqref{eq:sfPEhyp}, it is sufficient to consider the case $\varrho =0$. The proof follows the similar 
procedure as in the case of the conic surface, so we shall be brief. By  \eqref{eq:sfPEadd0Hyp} we can 
write $\sL_n^E(\sw_{0,\g}^0)$ in terms of the kernel for the Jacobi polynomials. Let $\l = \g+\frac{d -1}{2}$. 
Then
\begin{align*}
\sL_n^E (\sw_{0,\g}^0; (x,t), (y,s) ) =  c_{\g} \int_{-1}^1 & L_n ^{(\l-\f12,\l-\f12)} 
        \big(\zeta(x,t,y,s;v) \big)  (1-v^2)^{\g-1} \d v.
\end{align*}
Applying \eqref{eq:DLn(t,1)} with $m=0$ and $\alpha = \b =\l-1/2$, we then obtain
\begin{align*}
 \left| \sL_n^E (\sw_{0,\g}^0; (x,t), (y,s)) \right| \le c n^{2 \l +1} \int_{-1}^1
     & \frac{1}{ \left(1+ n \sqrt{1- \zeta (x,t,y,s; v)^2}\right)^{\k+3\g+1}}   (1-v^2)^{\g-1} \d v.
\end{align*}
By the definition of $\sd_{\XX_0}(\cdot,\cdot)$, we have 
\begin{align} \label{eq:1-xiXX0}
  1- \zeta(x,t,y,,s;t)  = 1- \cos \sd_{\XX_0} ((x,t), (y,s)) +  (1-v) \sqrt{1-t^2}\sqrt{1-s^2}.
\end{align}
In particular, $1- \zeta(x,t,y,,s;t)$ is bounded below by either the first term or the second term in
the right-hand side of \eqref{eq:1-xiXX0}, which leads to, in particular, the estimate 
\begin{align*}
 \left| \sL_n^E (\sw_{0,\g}^0; (x,t), (y,s) )\right| \, &\le c n^{2 \l +1} 
        \frac{1}{ \big(1+ n  \sd_{\XX_0} ((x,t), (y,s))\big)^{\k+\g}}  \\
     & \times  c_\g \int_{0}^1 
                     \frac{(1-v^2)^{\g-1}} {\left(1+n\sqrt{(1-v) \sqrt{1-t^2}\sqrt{1-s^2}}\right)^{2\g+1}} \d v,
\end{align*}
where we have used $1-\cos \t \sim \t^2$ and the symmetry of the integral. The last 
integral is evidently bounded by 1 and it can be estimated by using the inequality \cite[(13.5.8)]{DaiX} 
\begin{equation}\label{eq:B+At}
  \int_0^1 \frac{(1-t)^{a-1} \d t}{(1+n \sqrt{B+A(1-t)})^b} \le c \frac{n^{-2 a}}{A^a (1+n\sqrt{B})^{b-2a-1}},
\end{equation}
which holds for $A>0$, $B\ge 0$, $a>0$ and $b\ge 2 a +1$, which leads to the estimate 
\begin{align*}
 c_\g \int_{0}^1 & \frac{(1-v^2)^{\g-1}} {\left(1+n\sqrt{(1-v) \sqrt{1-t^2}\sqrt{1-s^2}}\right)^{2\g+1}} \d v 
       \le c \frac{n^{- 2 \g}}{ \left(\sqrt{1-t^2}\sqrt{1-s^2}+ n^{-1}\right)^\g} \\ 
     & \qquad \le c  \frac{n^{- 2 \g}}{\sqrt{ \sw_{0,\g}^0 (n; t)} \sqrt{\sw_{0,\g}^0 (n; t)}}
     (1+  n \mathsf{d}_{\XX_0}((x,t),(y,s)))^{\g},  
\end{align*}
where the second inequality follows from the elementary identity \cite[(11.5.13)]{DaiX}
\begin{equation} \label{eq:ab+=a+b+}
   (a+n^{-1})(b+n^{-1}) \le 3 (ab+n^{-2})(1+n|b-a|)
\end{equation}
with $a = \sqrt{1-t^2}$ and $b=\sqrt{1-s^2}$ and Lemma \ref{lem:|s-t|X0}. Putting the last two displayed 
inequalities together completes the proof. 
\end{proof}

This theorem establishes the Assertion 1 on $\XX_0^{d+1}$. We now turn to Assertion 2. 

\begin{thm} \label{thm:L-LkernelX0}
Let $d\ge 2$ and $\g \ge 0$. For either $(x_i,t_i), (y,s)$ all in $\XX_{0,+}^{d+1}$ or all in $\XX_{0,-}^{d+1}$, 
and $(x_1,t_1) \in \sc_\varrho \big((x_2,t_2), c^* n^{-1}\big)$ with $c^*$ small and for $\k > 0$,  
\begin{align}\label{eq:L-LkernelX0}
     &  \left |\sL_n^E (\sw_{0,\g}^\varrho; (x_1,t_1), (y,s))-\sL_n^E (\sw_{0,\g}^\varrho; (x_2,t_2), (y, s))\right| \\
     & \qquad \qquad
\le \frac{c_\k n^{d+1} \sd_{\XX_0}^\varrho ((x_1,t_1),(x_2,t_2))}{\sqrt{ \sw_{0,\g}^\varrho (n; t_2) }
   \sqrt{ \sw_{0,\g}^\varrho (n; s) } \big(1 + n \sd_{\XX_0}^\varrho ( (x_2,t_2), (y,s)) \big)^{\k}}. \notag
\end{align}
\end{thm}

\begin{proof}
Again, it is sufficient to consider $\varrho =0$. Denote the left-hand side of \eqref{eq:L-LkernelX0} by $K$. 
Let $\partial L(u) = L'(u)$. Using the integral expression of $\sL_n^E(\sw_{-1,\g})$, we obtain 
\begin{align} \label{eq:L-LkernelX0}
K & \le 2 \int_{-1}^1 \big\| \partial L_n^{\l-\f12, \l-\f12}\big\|_{L^\infty(I_v)} 
     \big |\zeta_1(v) - \zeta_2(v)|  (1-v^2)^{\g-1} \d v, \notag
\end{align}
where $\zeta_i(v) = \zeta (x_i,t_i,y,s;v)$, and $I_v$ is the interval with end points $\zeta_1(v)$ 
and $\zeta_2(v)$. We claim that 
\begin{equation}\label{eq:zeta1-zeta2}
   |\zeta_1(v) - \zeta_2(v)| \le \sd_{\XX_0}\big((x_1,t_1), (x_2,t_2)\big) \big[ \Sigma_1 + \Sigma_2 (v) \big], 
\end{equation}
where 
\begin{align*}
   \Sigma_1 \,& =  \sd_{\XX_0}((x_2,t_2), (y,s)) +\sd_{\XX_0}((x_1,t_1), (x_2,t_2)), \\
   \Sigma_2(v) \,& = (1-v) \sqrt{1-s^2}.
\end{align*}
To see this, we first use \eqref{eq:1-xiXX0} to write that 
\begin{align*}
   \zeta_1(v) - \zeta_2(v) = \,& \cos \sd_{\XX_0}((x_1,t_1),(y,s)) - \cos \sd_{\XX_0}((x_2,t_2),(y,s)) \\
              &  + (1-v) \left(\sqrt{1-t_1^2} - \sqrt{1-t_2^2}\right) \sqrt{1-s^2}. 
\end{align*}
Denote temporarily $\a_i =  \sd_{\XX_0}((x_1,t_1),(y,s))$ for $i =1, 2$. The identity 
\begin{align*}
  \cos \a_1 -  \cos \a_2 
    =   2 \sin \frac{\a_1 - \a_2}{2} \left( 2 \sin \frac{\a_1}{2} \cos \frac{\a_2}{2} - \sin \frac{\a_2-\a_1}{2}\right),
\end{align*} 
implies that $|\cos \a_1 -  \cos \a_2| \le |\a_1-\a_2| \left ( |\a_1| + \tfrac 12 |\a_1 - \a_2| \right)$, which leads
to the estimate for the $\Sigma_1$ term by the triangle inequality of $\sd_{\XX_0}$ and Lemma \ref{lem:|s-t|X0}.
The estimate for the $\Sigma_2$ terms is trivial. Hence, \eqref{eq:zeta1-zeta2} holds as claimed. 

Since $\max_{r\in I_v} (1+n \sqrt{1- r^2})^{-\k}$ is attained at one of the end points of the interval, it follows 
from \eqref{eq:DLn(t,1)} with $m =1$ and $\l = \g+ \frac{d-1}{2}$ that 
\begin{align*}
& \left |\sL_n^E (\sw_{0,\g}^0; (x_1,t_1), (y,s))-\sL_n^E (\sw_{-0,\g}^0; (x_2,t_2), (y, s))\right| \\
 &   \quad \le c \, \sd_{\XX_0}\big((x_1,t_1),(x_2,t_2)\big) \int_{-1}^1 \left[  \frac{n^{2 \l + 3}}
    {\big(1+n\sqrt{1-\zeta_1(v)^2} \big)^{\k(\g)}} +  \frac{n^{2 \l + 3}}{\big(1+n\sqrt{1-\zeta_2(v)^2}\big)^{\k(\g)}} \right ] \\
 &   \qquad\quad  \qquad\quad  \qquad\quad  \qquad\quad  \qquad\quad \times \big(\Sigma_1+\Sigma_2(v)\big)  (1-v^2)^{\g-1} \d v,
\end{align*}
where we choose $\k(\g) = \k+3 \g +2$. Since $(x_1,t_1) \in \sc \big((x_2,t_2), c^* n^{-1}\big)$, $\Sigma_1$ is bounded by 
$\Sigma_1 \le c n^{-1}  \big(1 + n \sd_{\VV_0}\big((x_i, t_i),(y,s)\big) \big)$. Hence, using the two lower bound 
of $1- \zeta_i(v)$ given by the right-hand side of \eqref{eq:1-xiXX0}, we obtain 
\begin{align*}
 \int_{-1}^1 & \frac{n^{2 \l + 3}}{\big(1+n\sqrt{1-\zeta_i(v)^2} \big)^{\k(\g)}} \Sigma_1 (1-  v^2)^{\g-1} \d v\\
 & \le  \frac{c \, n^{2 \l + 2} }{\big(1+n \sd_{\VV_0}((x_i,t_i),(y,s))\big)^{\k + \g}} \int_{0}^1 
                     \frac{(1-v^2)^{\g-1}} {\left(1+n\sqrt{(1-v) \sqrt{1-t_i^2}\sqrt{1-s^2}}\right)^{2\g+1}} \d v\\
 & \le   \frac{c_\k n^{d+1}} {\sqrt{ \sw_{0,\g}^0 (n; t_i) }
   \sqrt{ \sw_{0,\g}^0 (n; s) } \big(1 + n \sd_{\XX_0} ( (x_i,t_i), (y,s)) \big)^{\k}},
\end{align*} 
where the second step follows from the estimate of the last integral in the proof of Theorem \ref{thm:kernelX0}. 
Since  $\sw_{0,\g}^0 (n,t_1) \sim \sw_{0,\g}(n,t_2)$ and $\sd_{\XX_0}((x_1, t_1),(y,s))
 + n^{-1} \sim \sd_{\XX_0}((x_2, t_2),(y,s)) + n^{-1}$ by Lemma \ref{lem:|s-t|X0}, we can replace $(x_1,t_1)$ 
in the right-hand side by $(x_2,t_2)$. This shows that the integral containing $\Sigma_1$ has the desired estimate.

For the integral that contains $\Sigma_2(v)= (1-v)\sqrt{1-s^2}$, the factor $1-v$ increases the power of the 
weight to $(1-v_1)^{\g}$, so that we can follow the estimate for the integral with $\Sigma_1$ but using $(1-v)^\g$,
which leads to 
\begin{align*}
 \int_{-1}^1 & \frac{n^{2 \l + 3}}{\big(1+n\sqrt{1-\zeta_i(v)^2} \big)^{\k(\g+\f12)}} \Sigma_2(v_1) (1-  v^2)^{\g-1} \d v\\
 & \le  c  \frac{  n^{d+1} n^{-1}\sqrt{1-s^2} }{\sqrt{ \sw_{0,\g+\f12}^0(n; s) }\sqrt{ \sw_{0,\g+\f12}^0(n; t_i)}
   \big(1 + n \sd_{\VV_0}( (y,s), (x_i, t_i)) \big)^\k} \\
  & \le  c  \frac{  n^{d+1}  }{\sqrt{ \sw_{\g,d}^0 (n; s) }\sqrt{ \sw_{\g,d}^0 (n; t_2) }
     \big(1 + n \sd_{\VV_0}( (y,s), (x_i, t_i)) \big)^\k}, 
\end{align*} 
where the last step uses the inequality $n^{-1} \sqrt{1-s^2}\le (\sqrt{1-t_i^2}+n^{-1}) (\sqrt{1-s^2}+n^{-1})$. This 
takes care of the integral with $\Sigma_2(v)$ and completes the proof.
\end{proof}

The case of $p=1$ of the following lemma establishes Assertion 3 for $\sw_{\b,\g}^\varrho$. 

\begin{lem}\label{lem:intLnX0}
Let $d\ge 2$, $\b> -\f12$ and $\g > - \f12$. For $0 < p < \infty$, assume 
$\k > \frac{2d+2}{p} + 2(\b+\g) |\f1p-\f12|$. Then for $(x,t) \in \XX_0^{d+1}$,  
\begin{align}\label{eq:intLnX0}
\int_{\XX_0^{d+1}} \frac{ \sw_{\b,\g}^\varrho(s)}{  \sw_{\b,\g}^\varrho (n; s)^{\f{p}2}
    \big(1 + n \sd_{\XX_0}^\varrho( (x,t), (y,s)) \big)^{\k p}}   \d \s_\varrho(y,s) 
    \le c n^{-d} \sw_{\b,\g}^\varrho (n; t)^{1-\f{p}{2}}.
\end{align}
\end{lem}
 
\begin{proof}
We again only need to consider $\varrho = 0$. Let $J_p$ denote the left-hand side of \eqref{eq:intLnX0}. 
As shown in the proof of \cite[Lemma 2.4]{X21}, it is sufficient to prove the case $ p= 2$. Furthermore, the 
integral over $\XX_0^{d+1}$ is a sum of two integrals over $\XX_{0,+}^{d+1}$ and $\XX_{0,-}^{d+1}$, respectively. 
We only need to estimate one of them. Denote the integral over $\XX_{0,+}^{d+1} = \VV_0^{d+1}$ by $J_{2,+}$. 
Then 
\begin{align*}
J_{2,+}  = \int_{\VV_0^{d+1}} \frac{\sw^0_{\b,\g}(s)  } {\sw^0_{\b,\g} (n; s) 
      (1 + n \sd_{\XX_0}( (x,t), (y,s)) )^{2\k}}  \d \s_0(y,s).
\end{align*} 
Let $x = t \xi$ and $y = s\eta$. Using \eqref{eq:intSS}, we obtain 
\begin{align*}
 J_{2,+}  \, &\le c  \int_0^1 \int_{-1}^1 \frac{ s^{d-1}  \sw_{\b,\g}^0(s) (1-u^2)^{\f{d-3}{2}} }{\sw^0_{\b,\g}(n; s)
        \left(1 + n \sqrt{1-  t s u - \sqrt{1-t^2}\sqrt{1-s^2}}\right)^{2\k} } \d u  \d s.
\end{align*}
Hence, using $\sw^0_{\b,\g}(s) \le c \sw^0_{\b,\g}(n; s) (1-s^2+n^{-2})^{-\f12}$, making another change of
variable $u \mapsto v /s$ and simplifying, it follows that 
\begin{align*}
  J_{2,+}  \,& \le c   \int_0^1 \int_{-s}^{s} \frac{s\, (s^2- v^2)^{\f{d-3}{2}}  }{(1-s^2+n^{-2})^{\f12}
         \left(1 + n \sqrt{1-  t v - \sqrt{1-t^2}\sqrt{1-s^2}}\right)^{2\k} } \d v  \d s \\
        & \le c n^d  \int_{-1}^1 \int_{|v|}^{1}  \frac{ (s^2- v^2)^{\f{d-3}{2}}  }
           {\left(1 + n \sqrt{1-  t v - \sqrt{1-t^2}\sqrt{1-s^2}}\right)^{2\k} } \d s  \d v,        
\end{align*}
where we changed the order of integration in the second step. A further change of variable
$s \mapsto \sqrt{1-\|u\|^2}$ shows then
\begin{align*}
  J_{2,+}  \le c    \int_{-1}^1 \int_{0}^{\sqrt{1-v^2}}  \frac{ (1-u^2- v^2)^{\f{d-3}{2}}}
           {\left(1 + n \sqrt{1-  t v - \sqrt{1-t^2} u}\right)^{2\k} } \d u  \d v,        
\end{align*}
which is an integral over the right half $\{(u,v) \in \BB^2: v \ge 0\}$ of the unit disk $\BB^2$. Setting 
$z = t  v + \sqrt{1-t^2} \, u$ and $w = - \sqrt{1-t^2} v + t u$ in the integral, which is an orthogonal 
transformation, and enlarging the integral domain while taking into account that $z \ge 0$, it follows that 
\begin{align*}
  J_{2,+} \, & \le c    \int_0^1 \frac{1} {\left(1 + n \sqrt{1- z} \right)^{2\k} } 
        \int_{- \sqrt{1-z^2}}^{\sqrt{1-z^2}}(1-z^2- w^2)^{\f{d-3}{2}} \d w  \d z \\
       & \le c   \int_0^1 \frac{(1-z^2)^{\f{d-2}{2}} } {\left(1 + n \sqrt{1- z} \right)^{2\k} }  
          \le c  n^{-d} \int_0^n \frac{ r^{d-1} }{ (1 + r )^\k } \d r \le c n^{-d}
\end{align*} 
by setting $r = n \sqrt{1-z}$ and recalling that $\k > d$. This completes the proof.  
\end{proof}
 
\begin{prop}\label{prop:intLnX0}
For $\g \ge 0$ and $(x,t) \in \XX_0^{d+1}$,  
\begin{equation*}
   \int_{\XX_0^{d+1}} \left| \sL_n^E \big(\sw_{0,\g}^\varrho; (x,t), (y,s)\big) \right|^p \sw_{0,\g}^\varrho(s) 
       \d \s_\varrho(y,s) \le \bigg (\frac{n^d}{\sw_{0,\g}^\varrho(n;t)} \bigg)^{p-1}. 
\end{equation*}
\end{prop}

This follows by applying Lemma \ref{lem:intLnX0} on the estimate of Theorem \ref{thm:kernelX0}.

We have established Assertions 1 and 3 for $\sL_n^E( \sw_{\b,\g}^\varrho;\cdot,\cdot)$ and also Assertion 2
for $\sL_n^E( \sw_{0,\g}^\varrho;\cdot,\cdot)$. The kernel uses, however, only polynomials that are even in
$t$ and in $s$ variable. Consequently, we have proved the following: 

\begin{cor}
For $d\ge 2$, $\varrho \ge 0$ and $\g \ge 0$, the space $(\XX_0^{d+1}, \sw_{0,\g}^\varrho, \sd_{\XX_0}^\varrho)$ 
is a localizable homogeneous space, where its localized kernels are defined for polynomials even in $t$ and in
$s$ variables. 
 \end{cor}

\subsection{Maximal $\ve$-separated sets and MZ inequality} \label{sec:ptsX0}
We provide a construction of maximal $\ve$-separated set, as defined by Definition \ref{defn:separated-pts},
on the double conic and hyperbolic surfaces.  
 
We first consider the double conic surface; that is, $\varrho =0$. We shall need maximal $\ve$-separated sets on
the unit sphere. We adopt the following notation. For $\ve > 0$, we denote by $\Xi_{\SS}(\ve)$ a maximal 
$\ve$-separated set on the unit sphere $\sph$ and we let $\SS_\xi(\ve)$ be the subsets in $\sph$ so that the 
collection $\{\SS_\xi(\ve): \xi \in \Xi_\SS(\ve)\}$ is a partition of $\sph$, and we assume
\begin{equation}\label{eq:ptsV01} 
      \sc_{\SS}(\xi, c_1 \ve) \subset \SS_\xi(\ve) \subset \sc_{\SS}(\xi, c_2 \ve), \qquad \xi \in \Xi_{\SS}(\ve),
\end{equation} 
where $\sc_{\SS}(\xi,\ve)$ denotes the spherical cap centered at $\xi$ with radius $\ve$, $c_1$ and $c_2$ 
depending only on $d$. Such a $\Xi_\SS(\ve)$ exists for all $\ve > 0$, see for example \cite[Section 6.4]{DaiX}, 
and its cardinality satisfies  
\begin{equation}\label{eq:ptsV02} 
c_d' \ve^{-d+1} \le \# \Xi_{\SS}(\ve) \le c_d \ve^{-d+1}.
\end{equation} 

Let $\ve > 0$. We let $N = 2  \lfloor \frac{\pi}{2}\ve^{-1} \rfloor$, so that $N$ is an even integer. For $1\le j \le N$ 
we define 
$$
 \t_j:= \frac{(2j-1)\pi}{2 N},  \qquad \t_j^+ :=  \t_j- \frac{\pi}{2 N}  \quad \hbox{and} \quad \t_j^- :=  \t_j +\frac{\pi}{2 N}.
$$
Let $t_j =  \cos \t_j$, $t_j^- = \cos \t_j^-$ and $t_j^+=\cos \t_j^+$. Thus, $t_1^+ = 1$ and $t_N^- = -1$ and 
$$
 1>  t_1 > t_2 > \ldots > t_{\f{N}{2}} > t_{\f{N} 2}^- = 0 =  t_{\f{N} 2 +1}^+ > t_{\f{N} 2 +1} > \ldots > t_N > -1.
$$
In particular, $t_{j-1}^+ = t_j^-$ and we can partition $\XX_0^{d+1}$ as the disjoint union of 
$$
 \XX_0^{(j)}:=  \left\{(x,t) \in \XX_0^{d+1}:   t_j^- < t \le t_j^+ \right \}, \qquad 1 \le j \le N.
$$
Furthermore, the upper and lower surfaces $ \XX_{0,+}^{d+1} $ and  $\XX_{0,-}^{d+1}$ can be partitioned by 
$$
   \XX_{0,+}^{d+1} =  \bigcup_{j=1}^{N/2} \XX_0^{(j)}  \quad\hbox{and}\quad \XX_{0,-}^{d+1} =  \bigcup_{j= N/2 +1}^N    \XX_0^{(j)}.
$$
Let $\ve_j := \pi \ve/(2 t_j)$. Then $\Xi_\SS(\ve_j)$ is the maximal $\ve_j$-separated 
set of $\sph$ such that $\{\SS_\xi(\ve_j): \xi \in \Xi_\SS(\ve_j)\}$ is a partition $\sph = \bigcup_{\eta \in \Xi_\SS(\ve_j)} \SS_\eta(\ve_j)$, and $\# \Xi_\SS(\ve_j) \sim \ve_j^{-d+1}$. For each 
$j =1,\ldots, N$, we decompose $\XX_0^{(j)}$ by 
$$
 \XX_0^{(j)} =  \bigcup_{\xi \in  \Xi_\SS(\ve_j)} \XX_0(\xi, t_j), \quad \hbox{where}\quad 
 \XX_0(\xi,t_j):= \left\{(t\eta,t):  t_j^- < t \le t_j^+, \, \eta \in \SS_\xi(\ve_j) \right\}.
$$
Finally, we define the subset $\Xi_{\XX_0}$ of $\XX_0^{d+1}$ by
$$
   \Xi_{\XX_0} = \big\{(t_j \xi, t_j): \,  \xi \in \Xi_\SS(\ve_j), \, 1\le j \le N \big\}. 
$$

\begin{prop} \label{prop:subsetX0}
Let $\ve > 0$ and $N = 2 \lfloor \frac{\pi}{2} \ve^{-1} \rfloor$. Then $\Xi_{\XX_0}$ is a maximal 
$\ve$-separated set of $\XX_0^{d+1}$ and $\{\XX_0(\xi, t_j): (t_j \xi, t_j) \in \Xi_{\XX_0}\}$ is a partition 
$$
   \XX_0^{d+1} =  \bigcup_{(t \xi, t) \in \Xi_{\XX_0}} \XX_0(\xi, t) = 
    \bigcup_{j=1}^N \bigcup_{\xi \in \Xi_\SS(\ve_j)} \XX_0(\xi,t_j).
$$
Moreover, there are positive constants $c_1$ and $c_2$ depending only on $d$ such that 
\begin{equation}\label{eq:incluX0cap}
      \sc_0 \big((t_j\xi,t_j), c_1 \ve\big) \subset \XX_0(\xi,t_j) \subset \sc_0 \big( (t_j \xi,t_j), c_2 \ve\big), 
\end{equation}
and $c_d'$ and $c_d$ such that 
\begin{equation}\label{eq:incluX0cap2}
c_d' \ve^{-d} \le \# \Xi_{\XX_0} \le c_d \ve^{-d}. 
\end{equation}
\end{prop}

\begin{proof}
Let $(t_j \xi, t_j)$ and $(t_k \eta, t_k)$ be two distinct points in $\Xi_{\XX_0}$. By its definition, 
$\sd_{[-1,1]}(t_j,t_k) = |\t_j-\t_k| \ge \frac{\pi}{N} \ge \ve$ if $j \ne k$. Hence, 
$$
   \sd_{\XX_0}\big((t_j \xi, t_j), (t_k \eta, t_k)  \big) \ge \sd_{[0,1]}(t_j,t_k)   \ge \ve,\quad j \ne k. 
$$
If $j = k$, then $\xi$ and $\eta$ are both elements of $\SS(\ve_j)$, so that $\sd_{\SS}(\xi,\eta) \ge \ve_j$. 
Hence, using $\f{2}{\pi}\phi \le \sin \phi \le \phi$, we deduce from \eqref{eq:d=d+dX0} that 
$$
\sd_{\XX_0}\big((t_j \xi, t_j), (t_j \eta, t_j)\big) \ge \frac{2}{\pi} t_j \sd_{\SS}(\xi,\eta) \ge \frac{2}{\pi} 
     t_j \ve_j = \ve. 
$$
Hence, $\Xi_{\XX_0}$ is $\ve$-separated. Moreover, since $\#\Xi_\SS(\ve_j) \sim \ve_j^{-d+1}$,
$$
 \#\Xi_{\XX_0} = \sum_{j=1}^N \#\Xi_\SS(\ve_j) 
     \sim \sum_{j=1}^N \ve_j^{-d+1}  \sim \ve^{-d+1} \sum_{j=1}^N t_j^{d-1} \sim \ve^{-d+1} N \sim \ve^{-d}. 
$$
For the proof of \eqref{eq:incluX0cap}, we only need to consider $\XX_0(\xi,t_j)$ in the upper part of the 
double cone, which means $1 \le j \le N/2$. If $\sd_{[-1,1]}(s,t_j) \le \delta/N$ with $\delta \le 1/2$, then 
$$
  |t_j - s| \le \sd_{[-1,1]} (t_j,s) \le \frac{\delta}{N} \le \delta \sin \f{\pi}{2N} = \delta t_{\f N 2} \le \delta t_j.
$$
Hence, by $\delta \le \f12$, we obtain $s \ge t_j /2$. Similarly, we see that if $s \in\sc_{[-1,1]}(t_j, \pi/N)$, 
then $s \le c_* t_j$. By \eqref{eq:ptsV01}, there are constants $b_1 > 0$ and $b_2> 0$ such that 
$\sc_\SS(\xi,b_1\ve_j)\subset \SS_\xi(\ve_j) \subset \sc_\SS(\xi,b_2 \ve_j)$. We claim that \eqref{eq:incluX0cap} 
holds for some $c_1 < \delta$ and some $c_2 > b_2$. Indeed, if $(y,\eta) \subset \sc_0\big((t_j\xi,t_j), c_1 \ve\big)$, 
then $\sd_{[-1,1]}(s,t_j) \le c_1 \ve \le \delta/N$; moreover, by $s \ge t_j/2$ and $(s t)^{\f12} \sd_{\SS}(\xi,\eta) 
\le c c_1 \ve$ by \eqref{eq:d2=d2+d2X0}, we see that $\sd_\SS(\xi,\eta) \le 2^{\f14} c c_1 \ve/\sqrt{t_j} \le b_1 \ve_j$ by choosing $c_1$ small. This establishes the left-hand side inclusion of \eqref{eq:incluX0cap}. The right-hand side 
inclusion can be similarly established. The proof is completed. 
\end{proof}

For $\varrho > 0$, the point set on the hyperbolic surface ${}_\varrho\XX_0^{d+1}$ can be 
deduced easily from that on the double conic  surface. 

\begin{prop}
For $\ve > 0$, let $\Xi_{\XX_0}$ be a maximal $\ve$-separated set in $\XX_{0}^{d+1}$. 
Define 
$$
\Xi_{\XX_{0}}^\varrho: = \left\{\left(x,\sqrt{t^2-\varrho^2}\right): (x,t) \in \Xi_{\XX_0}\right \}.
$$
Then $\Xi_{\XX_0}^\varrho$ is a maximal $\ve$-separated set in ${}_\varrho\XX_0^{d+1}$.
\end{prop}

This is an immediate consequence of \eqref{eq:distXX0rho}. In particular, if $\Xi_{\XX_0}$ is the
set given in Proposition \ref{prop:subsetX0}, then we can define $\XX_0^\varrho(\xi,t_j)$ accordingly 
so that both \eqref{eq:incluX0cap} and \eqref{eq:incluX0cap2} hold. 

\begin{defn}\label{def:evensymmetry}
Let $\Xi_{\XX_0}^\varrho$ be a set on $\XX_0^{d+1}$ that does not contain $(0,0)$. Define 
$$
 \Xi_{\XX_0, +}^\varrho = \{(x,t) \in \Xi_{\XX_0}^\varrho: t > 0\} \quad\hbox{and}\quad 
 \Xi_{\XX_0, -}^\varrho = \{(x,t) \in \Xi_{\XX_0}^\varrho: t < 0\}.
$$
We call the set $\Xi_{\XX_0}^\varrho$ evenly symmetric on $\XX_0^{d+1}$ if 
$$
     \Xi_{\XX_0, -}^\varrho =  \left \{(-x,-t): (x,t) \in \Xi_{\XX,+}^\varrho\right\}.
$$
\end{defn}

By definition, $\Xi_{\XX_0} =  \Xi_{\XX_0, +} \cup  \Xi_{\XX_0, -}$ and the two subsets are
disjoint. For the set $\Xi_{\XX_0}$ in Proposition \ref{prop:subsetX0}, the points $t_j$, $1\le j\le N$, 
are zeros of the Chebyshev polynomial of the first kind defined by $T_N(t) = \cos N \arccos (x)$ and we have 
$$
   \Xi_{\XX_0, +} =  \big\{(t_j \xi, t_j)\in \Xi_{\XX_0}: 1\le j \le \tfrac{N}{2} \big\},\quad 
   \Xi_{\XX_0, -} =  \big\{(t_j \xi, t_j)\in \Xi_{\XX_0}:  \tfrac{N}{2}+1 \le j \le N \big\}.
$$
Since $t_j = - t_{N-j}$ for $0 \le j \le \frac{N}2$, it follows that the set in Proposition \ref{prop:subsetX0} is
evenly symmetric on $\XX_0^{d+1}$, so is its analogue on ${}_\varrho\XX_0^{d+1}$.  

We further notice that, for the set $\Xi_{\XX_0}$ in Proposition \ref{prop:subsetX0}, with either $+$ or
$-$, $\Xi_{\XX_0, \pm}$ is a maximal $\ve$-separated sets of $\XX_{0,\pm}^{d+1}$ and the set 
$\{\XX_0(\xi, t_j): (t_j \xi, t_j) \in \Xi_{\XX_0, \pm} \}$ is a partition of $\XX_{0,\pm}^{d+1}$. Comparing
with the maximal $\ve$-separated set $\Xi_{\VV_0}$ constructed on $\VV_0^{d+1}=\XX_{0,+}^{d+1}$ 
in \cite{X21}, we see that the points in $\Xi_{\VV_0}$  congest towards the apex, with a rate $t_j \sim N^{-2}$, 
whereas the points in $\Xi_{\XX_0}$ do not. 
 
We have shown that $(\XX_0^{d+1}, \sw_{0,\g}^\varrho,\sd_{\XX_0}^\varrho)$ is a localizable homogeneous 
space with the highly localized kernels $\sL_n^E( \sw_{0,\g}^\varrho; \cdot,\cdot)$. As part of the framework
in \cite{X21}, we can then state the Marcinkiewicz-Zygmund inequality for a doubling weight on $\XX_0^{d+1}$, 
which holds under the following constrains: the weight function $\sw$ need to be even in $t$ variable
so that the integral of polynomials in $\Pi_n^E(\XX_0^{d+1})$ can be written as over $\XX_{0,+}^{d+1}$, the
maximal $\f \delta n$-separated set $\Xi_{\XX_0}$ need to be symmetric, and it works for polynomials
even in $t$ variable. Let us define, for $n=0,1,2,\ldots$, 
$$
\Pi_n^E(\XX_0^{d+1}) = \left\{p \in \Pi_n(\XX_0^{d+1}): p(x,t) = p(x,-t), \forall (x,t)\in \XX_0^{d+1}\right\}.
$$

\begin{thm} \label{thm:MZinequalityX0}
Let $\sw$ be an doubling weight on $\XX_0^{d+1}$ such that $\sw(x,t) = \sw(x,-t)$. Let 
$\Xi_{\XX_0}^\varrho$ be a symmetric maximal $\f \delta n$-separated subset of $\XX_0^{d+1}$ 
and $0 < \delta \le 1$. 
\begin{enumerate}[$(i)$]
\item For all $0<p< \infty$ and $f\in\Pi_m^E(\XX^{d+1})$ with $n \le m \le c n$,
\begin{equation*}
  \sum_{z \in \Xi_{\XX_0}} \Big( \max_{(x,t)\in \sc_\varrho ((z,r), \f \delta n)} |f(x,t)|^p \Big)
     \sw\!\left(\sc_\varrho ((z, r), \tfrac \delta n) \right) \leq c_{\sw,p} \|f\|_{p,\sw}^p
\end{equation*}
where $c_{\sw,p}$ depends on $p$ when $p$ is close to $0$ and on the doubling constant of $\sw$.
\item For $0 < r < 1$, there is a $\delta_r > 0$ such that for $\delta \le \delta_r$, $r \le p < \infty$ and 
$f \in \Pi_n^E(\XX_0^{d+1})$,  
\begin{align*}
  \|f\|_{p,\sw}^p \le c_{\sw,r} \sum_{z \in\Xi}
       \Big(\min_{(x,t)\in \sc_\varrho \bigl((z,r), \tfrac{\delta}n\bigr)} |f(x,t)|^p\Big)
          \sw\bigl(\sc_\varrho ((z,r), \tfrac \delta n)\bigr)
\end{align*}
where $c_{\sw,r}$ depends only on the doubling constant of $\sw$ and on $r$ when $r$ is close to $0$.
\end{enumerate}
\end{thm}

This is a consequence of \cite[Theorem 2.15]{X21}; its proof remains valid for polynomials in 
$\Pi_m^E(\XX^{d+1})$ under the assumptions on symmetry. 

\subsection{Positive cubature rules}
The Marcinkiewicz-Zygmund inequality is used to establish the positive cubature rule in the general 
framework. In order to quantify the coefficients of the cubature rule, we will need Assertion 4. This is
given by the fast decaying polynomials on the hyperbolic surface given below. 

\begin{lem}\label{lem:A4X0}
Let $d\ge 2$ and $\varrho \ge 0$. For each $(x,t) \in \XX_0^{d+1}$, there is a polynomial $T_{(x,t)}^\varrho$ 
in $\Pi_n^E(\XX_0^{d+1})$ that satisfies
\begin{enumerate} [   (1)]
\item $T_{(x,t)}^\varrho(x,t) =1$, $T_{x,t}^\varrho(y,s) \ge c > 0$ if $(y,s) \in \sc_\varrho( (x,t), \f{\delta}{n})$, 
and for every $\k > 0$,
$$
   0 \le T_{(x,t)}^\varrho(y,s) \le c_\k \left(1+ \sd_{\XX_0}^\varrho\big((x,t),(y,s)\big) \right)^{-\k}, \quad (y,s) \in \XX_0^{d+1}.
$$
\item there is a polynomial $q(t)$ of degree $4n$ such that $q(t) T_{(x,t)}^\varrho$ is a polynomial of degree 
$5n$ in $(x,t)$ variables and $1 \le q_n(t) \le c$. 
\end{enumerate}
\end{lem}

\begin{proof}
For positive integer $n$, let $m = \lfloor \frac{n}{r} \rfloor +1$ and define 
$$
  S_n(\cos \t) = \left( \frac{\sin (m+\f12)\f\t 2} {(m+\f12) \sin \frac{\t}2} \right)^{2r}, \qquad 0 \le \t \le \pi.
$$
Then $S_n(z)$ is an even algebraic polynomial of degree at most $2 n$ and it satisfies
\begin{equation} \label{eq:fastSn}
  S_n(1) = 1, \qquad 0 \le S_n(\cos \t) \le c \big(1+ n \t\big)^{-2r}, \quad 0 \le \t \le \pi.
\end{equation}
To construct polynomials on hyperbolic surface, we define, for $(x,t), (y,s) \in \XX_0^{d+1}$,  
\begin{align*}
  S((x,t),(y,s)) \, & =  S_n\left(\la x,y\ra +\sqrt{1+\varrho^2-t^2}\sqrt{1+\varrho^2-s^2}\right) \\
    &  + S_n\left(\la x,y\ra - \sqrt{1+\varrho^2-t^2}\sqrt{1+\varrho^2-s^2}\right).
\end{align*}
Since $S_n$ is an even polynomial, it follows that $S((x,t),(y,s))$ is a polynomial of degree $n$ in either $(x,t)$ 
or $(y,s)$ variables. Moreover, since $\|y\|^2 = s^2-\varrho^2$, it also follows that $S((x,t), \cdot)$ is even in $s$ 
variable. Define
$$
 T_{(x,t)}^\varrho(y,s) = \frac{S((x,t),(y,s))}{1 + S_n(2t^2-2\varrho^2 -1)}, \qquad (y,s) \in \XX_0^{d+1}.
$$
Then $T_{(x,t)}^\varrho \in \Pi_n^E(\XX_0^{d+1})$ and $T_{(x,t)}^\varrho(x,t) = 1$. If $0 \le \t \le \frac{2\pi}{2m+1}$, 
then it follows from $ \sin \t \ge \f{2}{\pi} \t$ for $0 \le \t \le \pi/2$ that $S_n (\cos \t) \ge \left(\frac{2}{\pi}\right)^{2r}$. 
Hence, since $0 \le S_n(2t-1) \le 1$, it follows that 
$$
T_{(x,t)}^\varrho(y,s)  \ge \frac{S_n\left(\la x,y\ra +\sqrt{1+\varrho^2-t^2}\sqrt{1+\varrho^2-s^2}\right)}
   {1 + S_n(2t^2-2\varrho^2 -1)} \ge  \f12 \left(\frac{2}{\pi}\right)^{2r}
$$
for $(y,s) \in \sc((x,t), \frac{2\pi}{2m+1})$. Next, we use the estimate $S_n(z) \le c (1+ n \sqrt{1-t^2})^{ - 2r}$ for
$|t| \le 1$ to obtain an upper bound for $T_{(x,t)}$. Since $
1- \la x,y\ra +  \sqrt{1-t^2}\sqrt{1-s^2} \ge  1- \la x, y\ra - \sqrt{1-t^2}\sqrt{1-s^2},
$
we obtain 
\begin{align*}
 0 \le T_{(x,t)}^\varrho(y,s) \, & \le c \left(1+ n \sqrt{1 - \la x,y\ra - \sqrt{1-t^2}\sqrt{1-s^2}}\right)^{- 2r} \\
    & = c  \left(1+ n \sqrt{1 -\cos \sd_{\VV_0}((x,t),(y,s))}\right)^{-2r} \sim \left(1+ n \sd_{\XX_0} ( (x,t),(y, s))\right)^{-2r}
\end{align*}
using $1 - \cos \t \sim \t^2$. This completes the proof of item (1). The item (2) follows from setting 
$q(t) = 1+ S_n(2t^2-2\varrho^2 -1)$, which is a polynomial of degree $4n$ and $1\le q(t) \le c$. This 
completes the proof. 
\end{proof}

The lemma establishes Assertion 4 with a polynomial in $\Pi_n^E(\XX_0^{d+1})$. It allows us to follow
the general framework to bound the Christoffel function in $\Pi_n^E(\XX_0^{d+1})$. Let $\sw$ be a doubling 
weight function, even in $t$ variable, on $\XX_0^{d+1}$. Let 
\begin{align}\label{eq:ChristoffelFX0}
   \l_n^E(\sw;x,t): = \inf_{\substack{g(x,t) =1 \\ g \in \Pi_n^E(\XX_0^{d+1})}} \int_{\XX_0^{d+1}} 
         |g(x,t)|^2 \sw(x,t)  \d \s (x,t),
\end{align} 
which is the Christoffel function for the space $\Pi_n^E(\XX_0^{d+1})$. Let $\sK_n^E(\sw;\cdot,\cdot)$ denote 
the kernel of the $n$-th partial sum operator of the the series \eqref{eq:FourierX0}. Then
$$
\sK_n^E(\sw; (x,t),(y,s)) = \sum_{k=0}^n \sP_k^E(\sw;(x,t),(y,s)).
$$ 
By the proof of \cite[Theorem 4.6.6]{DX}, it is related to the Christoffel function by 
\begin{align*}
   \l_n^E(\sw; x,t) = \frac{1}{\sK_n^E(\sw; (x,t),(x,t))}, \qquad (x,t) \in \XX_0^{d+1}.
\end{align*}    
Using Lemma \ref{lem:A4X0}, we can adopt \cite[Propositions 2.17 and 2.18]{X21} to bound $\l_n^E(\sw)$.

\begin{cor} \label{cor:ChristFX0}
Let $\sw$ be a doubling weight function on $\XX_0^{d+1}$ such that $\sw(x,t)= \sw(x,-t)$ for all
$(x,t) \in \XX_0^{d+1}$. Then 
\begin{equation*}  
   \l_n^E \big(\sw; (x,t) \big)  \le c \, \sw \left(\sc_\varrho\left((x,t), \tfrac1n \right) \right).
\end{equation*}
Moreover, for $\g \ge 0$, 
$$
\l_n^E \big(\sw_{0,\g}^\varrho; (x,t) \big)  \ge c' \, \sw_{0,\g}^\varrho \left(\sc_\varrho\left((x,t), \tfrac1n \right) \right) 
   = c' n^{-d} \sw_{0,\g}^\varrho(n; t).
$$
\end{cor}

We can now state the positive cubature rule for the hyperbolic surface, which holds for polynomials in 
$\Pi_n^E(\XX_0^{d+1})$ under the assumption of symmetry for both the weight $\sw$ and the set 
$\Xi_{\XX_0}$. 

\begin{thm}\label{thm:cubatureX0}
Let $d \ge 2$ and $\varrho \ge 0$. Let $\sw$ be a doubling weight on $\XX_0^{d+1}$ such that 
$\sw(x,t)= \sw(x,-t)$ for all $(x,t) \in \XX_0^{d+1}$. Let $\Xi^\varrho$ be a symmetric maximum 
$\frac{\delta}{n}$-separated subset of $\XX_0^{d+1}$. There is a $\delta_0 > 0$ such that for 
$0 < \delta < \delta_0$ there exist positive 
numbers $\l_{z,r}$, $(z,r) \in \Xi^\varrho$, so that 
\begin{equation}\label{eq:CFX0}
    \int_{\XX_0^{d+1}} f(x,t) \sw(x,t)  \d \s(x,t) = \sum_{(z,r) \in \Xi^\varrho}\l_{z,r} f(z,r), \quad 
            \forall f \in \Pi_n^E(\XX_0^{d+1}). 
\end{equation}
Moreover, $\l_{z,r} \sim \sw\!\left(\sc_\varrho((z,r), \tfrac{\delta}{n})\right)$ for all $(z,r) \in \Xi^\varrho$. 
\end{thm}

This is \cite[Theorem 2.20]{X21} when the domain becomes $\XX_0^{d+1}$ and it remains valid under 
the assumptions on symmetry.

\subsection{Localized polynomial frame}
The localized polynomials are constructed using the highly localized kernel defined with a cut-off function $\wh a$ 
of type type (b) that satisfies 
\begin{align} \label{eq:a-frame}
\begin{split} 
 \wh a(t)\ge \rho > 0, \qquad & \mbox{if $t \in [3/5, 5/3]$},\\
 [\wh a(t)]^2 + [\wh a(2t)]^2 =1, \qquad & \mbox{if $t \in [1/2, 1]$.}
\end{split}
\end{align}
For the hyperbolic or double conic surface, we will also require symmetry for the weight and for the 
$\ve$-separated subset. Let $\varrho \ge 0$ and let $\sw$ be a doubling weight on $\XX_0^{d+1}$ 
and assume that it is even in $t$ variable. Let $\sL_n^E(\sw)*f$ denote the near best approximation operator 
from $\Pi_{2n}^E(\XX^{d+1})$ defined by 
\begin{equation}\label{eq:near-best}
   \sL_n^E(\sw) * f (x): = \int_{\XX_0^{d+1}}f(y,s) \sL_n^E(\sw; (x,t),(y,s)) \sw(y,s)  \d \s_\varrho(y,s).
\end{equation}
For $j =0,1,\ldots,$ let $\Xi_j^\varrho$ be a symmetric maximal $\frac \delta {2^{j}}$-separated 
subset in $\XX_0^{d+1}$, so that 
\begin{equation*}
  \int_{\XX_0^{d+1}} f(x,t) \sw(x,t)  \d \s_\varrho(x,t) = \sum_{(z,r) \in \Xi_j^\varrho} \l_{(x,r),j} f(z,r), 
     \quad f \in \Pi_{2^j}^E(\XX_0^{d+1}).  
\end{equation*}
For $j=1,2,\ldots,$ define the operator $F_j^\varrho(\sw)$ by
$$
F_j^\varrho(\sw) * f = \sL_{2^{j-1}}^E(\sw) * f
$$
and define the frame elements $\psi_{(z,r),j}$ for $(z,r) \in \Xi_j^\varrho$ by 
$$ 
   \psi_{(z,r),j}(x,t):= \sqrt{\l_{(z,r),j}} F_j^\varrho((x,t), (z,r)), \qquad (x,t) \in \XX_0^{d+1}. 
$$ 
Then $\Phi =\{ \psi_{(z,r),j}: (z,r) \in \Xi_j^\varrho, \, j =1,2,3,\ldots\}$ is a tight frame. Following 
\cite[Theorem 2.21]{X21} of the general framework, we have the following: 

\begin{thm}\label{thm:frameX0}
Let $\sw$ be a doubling weight on $\XX_0^{d+1}$ even in its $t$ variable. If 
$f\in L^2(\XX_0^{d+1}, \sw)$ and $f$ is even in $t$ variable, then
$$ 
   f =\sum_{j=0}^\infty \sum_{(z,r) \in\Xi_j^\varrho}
            \langle f, \psi_{(z,r), j} \rangle_\sw \psi_{(z,r),j}  \qquad\mbox{in $L^2(\XX_0^{d+1}, \sw)$}
$$  
and
$$ 
\|f\|_{2, \sw}  = \Big(\sum_{j=0}^\infty \sum_{(z,r) \in \Xi_j^\varrho} 
       |\langle f, \psi_{(z,r),j} \rangle_\sw|^2\Big)^{1/2}.
$$ 
Furthermore, for $\g \ge -\f12$, the frame for $\sw_{0,\g}^\varrho$ is highly localized in the sense that, 
for every $\s >0$, there exists a constant $c_\s >0$ such that 
\begin{equation} \label{eq.needleX0}
   |\psi_{(z,r),j}(x,t)| \le c_\s \frac{2^{j d/2}}{\sqrt{ \sw_{0,\g}^\varrho(2^{j}; t)} (1+ 2^j \sd_{\XX_0}^\varrho((x,t),(z,r)))^\s}, 
     \quad (x,t)\in \XX_0^{d+1}.
\end{equation}
\end{thm}
 
The frame elements involve only orthogonal polynomials even in $t$ variable and they are well defined 
for all doubling weight that is even in its $t$ variable. The localization \eqref{eq.needleX0} follows from 
Theorem \ref{thm:kernelX0} and $\l_{(z,r),j} \sim 2^{- jd} \sw_{0,\g}^\varrho(2^j;t)$, which holds for 
$\sw_{0,\g}^\varrho$ by Corollary \ref{cor:ChristFX0} and \eqref{eq:capX0}. It is worthwhile to point out 
that the localized frame is established for the Lebesgue measure on $\XX_0^{d+1}$, which is the case 
$\sw_{0,\f12}$, and in particular for the Lebesgue measure on the upper conic surface $\VV_0^{d+1}$, 
in contrast to the localized frame established in \cite[Section 2.7]{X21}. 
 
\subsection{Characterization of best approximation} 
For $f\in L^p(\XX_0^{d+1}, \sw_{0,\g}^\varrho)$, we denote by $\sE_n(f)_{p, \sw^\varrho_{0,\g}}$ the best 
approximation to $f$ from $\Pi_n(\XX_0^{d+1})$, the space of polynomials of degree at most $n$ restricted 
on the surface $\XX_0^{d+1}$, in the norm $\|\cdot\|_{p, \sw_{0,\g}^\varrho}$; that is, 
$$
    \sE_n(f)_{p, \sw_{0,\g}^\varrho}:= \inf_{g \in \Pi_n(\XX_0^{d+1})} \|f - g\|_{p,  \sw_{0,\g}^\varrho}, \qquad 1 \le p \le \infty.
$$
If $f$ is even in $t$ variable, then the triangle inequality and changing variable $t \mapsto -t$ shows, by
the symmetry of the integrals, 
$$
   \left\|f(x,t)- \frac 12 (g(x,t)+ g(x,-t))\right \|_{p,  \sw_{0,\g}^\varrho}\le \|f - g\|_{p,  \sw_{0,\g}^\varrho}.
$$
Hence, we can choose the polynomial of best approximation from $\Pi_n^E(\XX_0^{d+1})$ when $f$ is 
symmetric in $t$ variable. Following the general framework in \cite{X21}, we can give a characterization 
of best approximation by polynomials for functions even in $t$ variable. 

We define a $K$-functional and a modulus of smoothness. In terms of the fractional differential operator 
$(-\Delta_{0,\g}^\varrho)^{\f r 2}$ and a doubling weight $\sw$, even in $t$ variable on $\XX_0^{d+1}$, 
the $K$-functional fo $f\in L^p(\XX_0^{d+1}, \sw)$ and $r > 0$, is defined by
$$
   \sK_r(f,t)_{p,\sw} : = \inf_{g \in W_p^r(\XX_0^{d+1}, \sw)}
      \left \{ \|f-g\|_{p,\sw} + t^r\left\|(-\Delta_{0,\g}^\varrho)^{\f r 2}f \right\|_{p,\sw} \right \},
$$
where the Sobolev space $W_p^r\big(\XX_0^{d+1}, \sw\big)$ is the space that consists of $g \in L^p(\XX_0^{d+1},\sw)$,
even in $t$ variable, so that $\left\|(-\Delta_{0,\g}^\varrho)^{\f r 2}g \right\|_{p,\sw}$ is finite. The $K$-functional 
is well defined, as shown in \cite[Section 3.3]{X21}, where we need to require the functions being even in $t$ 
variable in the proof. Moreover, the modulus of smoothness is defined by
$$
  \o_r(f; \rho)_{p,\sw_{0,\g}^\varrho} = \sup_{0 \le \t \le \rho} 
     \left\| \left(I - \sS_{\t,\sw_{0,\g}^\varrho}\right)^{r/2} f\right\|_{p,\sw_{0,\g}^\varrho}, \quad 1 \le p \le \infty, 
$$
where the operator $\sS_{\t,\sw_{0,\g}^\varrho}$ is defined by, for $n = 0,1,2,\ldots$ and $\l = \g+\f{d-1}2$, 
$$
 \proj_n^E(\sw_{0,\g}^\varrho; \sS_{\t,\sw_{0,\g}^\varrho}f) = \frac{C_n^{\l-\f12} (\cos \t)}{C_n^{\l-\f12} (1)}
      \proj_n^E(\sw_{0,\g}^\varrho; f).
$$
The operator is well defined since the above relations determine it uniquely among functions even in 
$t$ variable and in $L^p(\XX_0^{d+1}, \sw_{0,\g}^\varrho)$ and the modulus of smoothness satisfies
the usual properties under its name. Such a modulus of smoothness is in line with
those defined on the unit sphere and the unit ball \cite{DaiX, Rus, X05} but its structure is more complicated. 
There are recent results for modulus of smoothness on fairly general domains in $\RR^d$ 
\cite{DP2,DP3,T1,T2}, defined via simple difference operators of functions, but conic domains are not 
included. 

By \cite[Theorem 3.1.2]{X21}, the characterization of the best approximation holds under the Assertions 1, 3 and 5.  
For $\sw_{0,\g}^\varrho$ on $\XX_0^{d+1}$, we have already established Assertions 1 and 3. We now establish
Assertion 5, again for kernels even in $t$ and $s$ variables. By \eqref{eq:sfHypGdiff}, the kernel 
$L_n^{(r)}(\varpi)$ in Assertion 5 becomes
$$
    \sL_n^{(r)}\big(\sw_{0,\g}^\varrho; (x,t),(y,s)\big) =\sum_{k=0}^\infty \wh a\left(\frac{k}{n} \right) (k(k+2\g+d-1))^{\f r 2} 
       \sP_k^E\big(\sw_{0,\g}^\varrho; (x,t),(y,s)\big).
$$

\begin{lem}\label{lem:6.71}
Let $\g \ge -\f12$ and $\k > 0$. Then, for $r > 0$ and $(x,t), (y,s) \in \XX_{0,+}^{d+1}$, 
$$
 \left| \sL_n^{(r)}\big(\sw_{0,\g}^\varrho; (x,t),(y,s)\big)\right| \le c_\k   
       \frac{n^{d+r}}{\sqrt{ \sw_{0,\g}^\varrho (n; t) }\sqrt{ \sw_{0,\g}^\varrho (n; s) }}
\big(1 + n \sd_{\XX_0}^\varrho( (x,t), (y,s)) \big)^{-\k}.
$$
\end{lem}

\begin{proof} 
By \eqref{eq:sfPEhyp} it suffices to consider the case $\varrho =0$. By \eqref{eq:sfPEadd0Hyp}, the 
kernel can be written as 
\begin{align*}
 \sL_n^{(r)}\big(\sw_{0,\g}^\varrho;(x,t),(y,s)) =  c_{\g}  \int_{-1}^1 L_{n,r}\left( \zeta(x,t,y,s;v)\right)(1-v^2)^{\g-1} \d v,
\end{align*}
where $L_{n,r}$ is defined by, with $\l = \g+\f{d -1}{2}$, 
$$
 L_{n,r}(t) = \sum_{k=0}^\infty \wh a\left(\frac{k}{n} \right) (k(k+2\g+d-1))^{\f r 2} 
     \frac{P_n^{(\l-\f12, \l-\f12)}(1)P_n^{(\l-\f12, -\l-\f12)}(t)}{h^{(\l-\f12,\l-\f12)}}.
$$
Applying \eqref{eq:DLn(t,1)} with $\eta(t) = \wh a(t) \left( t( t + n^{-1} (2\g+d-1))\right)^{\f r 2}$ and $m=0$, 
we obtain 
$$
 \left| L_{n,r}(t) \right| \le c n^{r} \frac{n^{2\l+1}}{(1+n\sqrt{1-t})^\ell}.  
$$
Using this estimate, it is easy to see that the proof follows from the estimate already established in 
the proof of Theorem \ref{thm:kernelX0}. 
\end{proof}
 
We are now in position to state the characterization of the best approximation by polynomials for the hyperbolic 
surface, following \cite[Theorem 3.12]{X21}. 

\begin{thm}\label{thm:bestappX0}
Let $f \in L^p(\XX_0^{d+1}, \sw)$ if $1 \le p < \infty$ and $f\in C(\XX_0^{d+1})$ if $p = \infty$. 
Assume that $f$ satisfies $f(x,t) = f(x,-t)$. Let $\varrho \ge 0$, $r > 0$ and $n =1,2,\ldots$. Then
\begin{enumerate} [   (i)]
\item For $\sw = \sw_{0,\g}^\varrho$ with $\g \ge 0$,  
$$
  \sE_n(f)_{p,\sw_{0,\g}^\varrho} \le c \,\sK_r (f;n^{-1})_{p,\sw_{0,\g}^\varrho}.
$$
\item for every doubling weight $\sw$ that satisfies $\sw(x,t) = \sw(x,-t)$ on $\XX_0^{d+1}$, 
$$
   \sK_r(f;n^{-1})_{p,\sw} \le c n^{-r} \sum_{k=0}^n (k+1)^{r-1}\sE_k(f)_{p, \sw}.
$$
\end{enumerate}
\end{thm}
 
For $\sw = \sw_{0,\g}$, the $K$-functional is equivalent to the modulus of smoothness. 

\begin{thm} \label{thm:K=omegaX0}
Let $\g \ge 0$ and $f \in L_p^r(\XX_0^{d+1}, \sw_{0,\g}^\varrho)$,  $1 \le p \le \infty$. Then
for $0 < \t \le \pi/2$ and $r >0$ 
$$
   c_1 \sK_r(f; \t)_{p,\sw_{0,\g}^\varrho} \le \o_r(f;\t)_{p,\sw_{0,\g}^\varrho} \le c_2 \sK_r(f;\t)_{p,\sw_{0,\g}^\varrho}.
$$
\end{thm}

Consequently, the characterization in Theorem \ref{thm:bestappX0} can be stated in terms of the 
modulus of smoothness in place of $K$-functional. 

Finally, we mention that, if $\wh a$ is an admissible cut-off function of type (a), then the operator
$\sL_n^E(\varpi) * f$ defined in \eqref{eq:near-best} is the near best approximation in the sense that
$$
     \|\sL_n^E(\sw_{0,\g}^\varpi) - f\|_{p, \sw} \le c \, \sE_n(f)_{p,\sw_{0,\g}^\varrho}, \quad 1 \le p \le \infty
$$
for all $f \in L_p^r(\XX_0^{d+1}, \sw)$, where $\sw$ is a doubling weight by \cite[Theorem 3.15]{X21}. 

\section{Homogeneous space on double cone and hyperboloid} 
\setcounter{equation}{0}

We work in the setting of homogeneous space on the solid domain defined by 
$$
  \XX^{d+1} =  \left \{(x,t): \|x\|^2 \le t^2 - \varrho^2, \, x \in \RR^d, \, \varrho \le |t| \le \sqrt{\varrho^2 +1}\right\}, 
$$
which is a double hyperboloid when $\varrho > 0$ and a double cone when $\varrho = 0$, and it is 
bounded by $\XX_0^{d+1}$ and the hyperplanes $t = \pm  \sqrt{\varrho^2 +1}$ of $\RR^{d+1}$. The 
analysis on $\XX^{d+1}$ differs substantially from that on the cone $\VV^{d+1}$ because the distance 
function is defined differently. We shall verify that the framework for homogeneous space is applicable 
on this domain for a family of weight functions related to the Gegenbauer weight and the classical 
weight on the unit ball, following closely the study on the hyperbolic surface. The structure of this section
is parallel to that of previous section, with contents arranging in the same order and under similar 
section names. Furthermore, part of the proof and development follows from the counterpart on 
$\XX_0^{d+1}$; hence, the proof is often brief or omitted. 

\subsection{Distance on the solid double cone and hyperboloid}

We write ${}_\varrho \XX^{d+1}$ whenever it necessary to emphasis the dependence on $\varrho$, but will
use $\XX^{d+1}$ most of the time. The domain $\XX_0^{d+1}$ can be decomposed as an upper part
and a lower part, 
$$
    \XX^{d+1}  = \XX_+^{d+1} \cup \XX_{-}^{d+1} =  \{(x,t) \in \XX^{d+1}: t \ge 0\} \cup
      \{(x,t) \in \XX^{d+1}: t \le 0\}. 
$$
For $\varrho =0$, the upper part is the solid cone $\XX_+^{d+1} = \VV^{d+1}$.

The distance function on the hyperboloid needs to take into account of the boundary behavior of the domain.
Like the case of the hyperbolic surface, the boundary in this case are the intersection of the surface with
the hyperplanes $t =1$ and $t= -1$. In particular, the apex point is not considered a boundary point. We first 
define the distance function on the double cone, that is, when $\varrho =0$.

\begin{defn}
Let $\varrho = 0$. For $(x,t), (y,s) \in \XX^{d+1}$, define
\begin{align*}
   \sd_{\XX}( (x,t),(y,s)) = \arccos \Big(  \la x,y\ra    +
       \sqrt{t^2-\|x\|^2} & \sqrt{s^2-\|y\|^2}  + \sqrt{1-t^2} \sqrt{1-s^2} \Big).
\end{align*}
Then $ \sd_{\XX}(\cdot,\cdot)$ is a distance function on the double cone $\XX^{d+1}$. 
\end{defn}

Let $X=\big(x, \sqrt{t^2 - \|x\|^2}\big)$ and $Y= \big(y, \sqrt{s^2-\|y\|^2}\big)$, then $|t|^{-1}X$ and 
$|s|^{-1}Y$ belong to $\BB^{d+1}$, so that $(X,t), (Y,s) \in \XX_0^{d+2}$ and 
\begin{equation}\label{eq:distX}
   \sd_{\XX^{d+1}}( (x,t),(y,s)) = \sd_{\XX_0^{d+2}} \big( (X,t), (Y,s)\big).
\end{equation}
In particular, it follows that $\sd_{\XX}(\cdot,\cdot)$ defines a distance on the solid double cone $\XX^{d+1}$. 

This distance function, however, is difference from the distance $\sd_{\VV}(\cdot,\cdot)$ defined in
for $\VV^{d+1}$ in \cite{X21}. It is closely related to the distance functions $\sd_{[-1,1]}(\cdot,\cdot)$ of $[-1,1]$ 
and the distance function on $\BB^d$ defined by
\begin{equation}\label{eq:distBall}
   \sd_{\BB}(x',y'):= \arccos \left( \la x',y'\ra + \sqrt{1-\|x'\|^2} \sqrt{1-\|y'\|^2} \right), \quad x',y' \in \BB^d.
\end{equation}

\begin{prop}\label{prop:cos-distXX}
Let $\varrho  =0$ and $d \ge 2$. For $(x,t), (y,s) \in \XX^{d+1}$, write $x =t x'$ and $y = s y'$ with 
$x', y' \in \BB^d$. Then
\begin{equation}\label{eq:d=d+dXX2}
   1- \cos \sd_{\XX} ((x,t), (y,s)) =1-\cos \sd_{[-1,1]}(t,s) + t s \left(1-\cos \sd_{\BB}(x',y') \right).
\end{equation}
In particular, if $t$ and $s$ have the same sign, then
\begin{equation}\label{eq:d<d+dXX}
 c_1 \sd_{\XX} ((x,t), (y,s))  \le \sd_{[-1,1]}(t,s) + \sqrt{t s} \, \sd_{\BB}(x',t')  \le c_2 
    \sd_{\XX} ((x,t), (y,s)).
\end{equation}
\end{prop}
 
The proof is similar to that of Proposition \ref{prop:cos-distX0}, using \eqref{eq:dist[-1,1]} and \eqref{eq:distBall}. 
In particular, it follows that the distance on the line segment $l_x = \{(t x', t): -1 \le t \le 1\}$, where $x' \in \BB^d$, 
of the double cone becomes $\sd_{[-1,1]}(t,s)$ as expected. We will also need the following lemma

\begin{lem} \label{lem:|s-t|XX}
Let $\varrho = 0$ and $d \ge 2$. For $(x,t), (y,s)$ either both in $\XX_+^{d+1}$ or both in $\XX_-^{d+1}$, 
$$ 
    \big| t-s \big| \le \sd_{\XX} ((x,t), (y,s)) \quad\hbox{and}\quad 
            \big| \sqrt{1-t^2} - \sqrt{1-s^2} \big| \le \sd_{\XX} ((x,t), (y,s)).
$$
and 
$$
      \big| \sqrt{t^2-\|x\|^2} - \sqrt{s^2 -\|y\|^2} \big| \le (\sqrt{2} + \pi) \sd_{\XX} ((x,t), (y,s)).
$$
\end{lem} 

\begin{proof}
We consider only $t ,s \ge 0$. Let $x =t x'$ and $y = s y'$ with $x', y' \in \BB^d$. Setting $t =\cos \t$ 
and $s = \cos \phi$. Using the inequality $|\la x',y'\ra + \sqrt{1-\|x'\|^2}\sqrt{1-\|y'\|^2}| \le 1$, it follows readily that 
$\cos \sd_{\XX_0} ((x,t), (y,s)) \le  t s + \sqrt{1-t^2} \sqrt{1-s^2} = \cos (\t-\phi)$, which allows us to follow
the proof of Lemma \ref{lem:|s-t|X0} to establish the first two inequalities. For third inequality, we assume
without loss of generality that $|t| \ge |s|$. Then,  
\begin{align*}
 \big| \sqrt{t^2-\|x\|^2} - \sqrt{s^2 -\|y\|^2} \big| & = \left| |t| \sqrt{1-\|x'\|^2} - |s| \sqrt{1-\|y'\|^2}\right| \\
   & \le |t-s| \sqrt{1-\|x'\|^2} + |s| \left| \sqrt{1-\|x'\|^2} - \sqrt{1-\|y'\|^2}\right| \\
   & \le   \sqrt{2} \sd_{\XX} ((x,t), (y,s)) +\sqrt{ t s}\, \sd_\BB(x',y'), 
\end{align*}
where the last step uses the inequality \cite[(A.1.4)]{DaiX} 
\begin{equation*} 
  \left| \sqrt{1-\|x'\|^2} -  \sqrt{1-\|y'\|^2}\right| \le \sqrt{2} \sd_\BB(x',y').  
\end{equation*}
Hence, the third inequality follows from \eqref{eq:d<d+dXX}.  This completes the proof. 
\end{proof}

For $\varrho > 0$, the two parts of the solid hyperboloid, $\XX_{+}^{d+1}$ and $\XX_{-}^{d+1}$, are
disjoint. It is sufficient to consider the distance between points that lie in the same part. We define
\begin{align} \label{eq:distXrho}
 & \sd_{\XX}^\varrho( (x,t), (y,s))= \\
  &\quad  \arccos \left (\la x,y\ra +\sqrt{1-\|x\|^2}\sqrt{1-\|y\|^2}
    +\sqrt{1+\varrho^2-t^2}\sqrt{1+\varrho^2-s^2} \right). \notag
\end{align}
If $(x,t)$ and $(y,t)$ are both in $\XX+^{d+1}$ or both in $\XX_{-}^{d+1}$, then 
\begin{align}\label{eq:distXXrho}
  \sd_{\XX}^\varrho( (x,t), (y,s)) = \sd_{\XX} \left( \Big(x, \sqrt{t^2-\varrho^2}\Big), \Big(y,\sqrt{s^2 - \varrho^2}\Big)\right). \end{align}
It is easy to see that this is a distance function and, evidently, its properties follows from those of the 
distance on the double cone. 

\subsection{A family of doubling weights}

For $d \ge 2$, $\b > -\tfrac{1}2$, $\g> -\f12$ and $\mu > -\f12$, let $W_{\b,\g,\mu}^\varrho$ be the
weight function defined on the hyperboloid $\XX^{d+1}$ by 
\begin{equation}\label{eq:SolidHypW}
 W_{\b,\g, \mu}^\varrho(x,t) = \bb_{\b,\g,\mu}^\varrho |t| (t^2-\varrho^2)^{\b-\f12} (1+\varrho^2-t^2)^{\g-\f12}(t^2-\varrho^2-\|x\|^2)^{\mu - \f12}.
 \end{equation}
When $\varrho = 0$ or the double cone, the weight function becomes
$$
     W_{\b,\g, \mu}^0(x,t) = \bb_{\b,\g,\mu}^0 |t|^{2\b} (1-t^2)^{\g-\f12}(t^2-\|x\|^2)^{\mu - \f12},
$$
which remains integrable over $\XX^{d+1}$ if $\b > - \frac{d+1}{2}$. Using the identity 
\begin{align} \label{eq:intw(t)X}
 \int_{{}_\varrho \XX^{d+1}} f (x,t) |t| \d x \d t & =
     \int_{\varrho \le |t|\le \varrho +1} |t| \int_{\|x\| \le \sqrt{t^2-\varrho^2}} f(x,t) \d  x \d t\\
      &    = \int_{|s|\le 1}|s| \int_{\|x\| \le |s|} f (x, \sqrt{s^2+\varrho^2}) \d y \d s\notag  \\
      &     =  \int_{{}_0\XX^{d+1}} f (x, \sqrt{s^2+\varrho^2})|s| \d x \d s, \notag
\end{align}
it is easy to see that the normalization constant $\bb_{\b,\g,\mu}^\varrho$ of $W_{\b,\g, \mu}^\varrho(x,t)$
satisfies
$$
\bb_{\b,\g,\mu}^\varrho = \bb_{\b,\g,\mu}^0 = b_\mu^\BB 
  \frac{\Gamma(\b+\mu+\g+\f{d+1}{2})}{\Gamma(\b+\mu\f d 2)\Gamma(\g+\f12)},
$$ 
where $b_\mu^\BB$ is the normalization constant for the weigt function $\varpi_\mu$ in \eqref{eq:weightB}
on $\BB^d$.

For $r > 0$ and $(x,t)$ on the solid hyperboloid ${}_\varrho \XX^{d+1}$, we denote the ball centered 
at $(x,t)$ with radius $r$ by 
$$
   \cb_\varrho((x,t), r): = \left \{ (y,s) \in \XX^{d+1}: \sd_{\XX}^\varrho \big((x,t),(y,s)\big)\le r \right\}.
$$   
The following lemma can be proved similarly as that of Lemma \ref{lem:cap-rhoX0}. 
\begin{lem}
For $\varrho > 0$ and $(x,t) \in  \XX^{d+1}$, 
$$
 W_{\b,\g,\mu}^\varrho \big(\cb_\varrho((x,t), r)\big) 
  = W_{\b,\g\mu}^0 \left(\cb_0\left(\big(x, \sqrt{t^2-\varrho^2}\big), r\right)\right). 
$$
\end{lem}

\begin{prop}\label{prop:capX}
Let $r > 0$ and $(x,t) \in \XX^{d+1}$.  Then for $\b > - \f{d+1}{2} $ and $\g > -1$ and $\mu \ge 0$, 
\begin{align}\label{eq:capX}
   W_{\b,\g,\mu}^0\big(\cb_0((x,t), r)\big) & :=\bb_{\b,\g,\mu}^0 \int_{ \cb_0((x,t), r)} W_{\b,\g,\mu}^0(y,s) \d y \d s \\
     &\, \sim r^{d+1} \left(t^2+ r^2\right)^{\b} \left(1-t^2+ r^2\right)^{\g} \left(t^2 - \|x\|^2+r^2\right)^\mu. \notag
\end{align}
In particular, $W_{\b,\g,\mu}^0$ is a doubling weight on the double cone and the doubling index
is given by $\a(W_{\b,\g,\mu}^0) = d+1+ 2\mu + 2 \max\{0,\b\}+2 \max\{0,\g\}$.
\end{prop}

\begin{proof} 
Let $\tau_r(t,s)$ and $\t_r(t,s) = \arccos \tau_r(t,s)$ be as in the proof of Proposition \ref{prop:capX0}. 
From $\sd_{\XX}((x,t),(y,s)) \le r$, we obtain $\sd_{[-1,1]}(t, s) \le r$ and, by \eqref{eq:distX}
$$
  \sd_{\BB} (x', y') \le \arccos \left(2 [\tau_r(t,s)]^2 -1\right) = \tfrac12 \arccos \tau_r(t,s) = \tfrac12 \t_r(t,s), 
$$
where $\sd_\BB(\cdot,\cdot)$ is the distance on the unit ball $\BB^d$. Hence, it follows that
\begin{align*}
 & W_{\b,\g,\mu}^0\big(\cb_0((x,t), r)\big)  = \bb_{\b,\g,\mu}^0 \int_{\sd_{[-1,1]}(t, s)\le r} s^{d} 
        \int_{\sd_{\BB}(x',y') \le \tfrac12 \t_r(t,s)} W_{\b,\g,\mu}^0(y,s)  \d y \d s \\ 
    & \quad \sim \int_{\sd_{[-1,1]}(t, s)\le r} s^{2\b+2\mu +d -1}(1-s^2)^{\g-\f12}  \d s 
          \int_{\sd_{\BB}(x',y') \le \tfrac12 \t_r(t,s)} (1-\|y'\|^2)^{\mu-\f12} \d y. 
\end{align*}
For $\mu \ge 0$ and $0 < \rho \le 1$, it is known \cite[Lemma 5.3]{PX2} or \cite[p. 107]{DaiX} that  
\begin{equation*}
 \int_{\sd_{\BB}(x',y') \le \rho} (1-\|y'\|^2)^{\mu-\f12} \d y' \sim (1-\|x'\|^2 + \rho^2)^\mu \rho^d,
\end{equation*}
which implies, together with $\t_r(t,s) \sim \sqrt{1-\tau_r,t,s}$, that 
\begin{align*}
  W_{\b,\g,\mu}^0\big(\cb_0((x,t), r)\big) 
     \sim & \int_{\sd_{[-1,1]}(t, s)\le r}   s^{2\b+2\mu +d -1} (1-s^2)^{\g-\f12} \\
         & \quad \times (1-\|x'\|^2+ 1 - \tau_r(t,s))^\mu (1 - \tau_r(t,s))^{\f d 2} \d s. 
\end{align*}
If $ t \ge 3 r$, then $s \sim t +r$ and, by $1-\tau_r(t,s) = (\cos \sd_{[-1,1]}(t,s) - \cos r)/(ts)$, it follows that 
$$
     s^{2\mu} \big(1-\|x'\|^2 + 1-\tau_r(t,s) \big)^\mu \sim  (t^2-\|x\|^2 + r^2)^\mu. 
$$
With this term removed, the integral of the remaining integrand can be estimated by following the 
estimates of Case 1 and Case 3 of the proof of Proposition \ref{prop:capX0}. If $t \le 3 r$, then 
$t^2 - \|x\|^2 + r^2 \sim r^2$. We use $1-\|x'\|^2 + 1-\tau_r(t,s) \le 2$ for the upper bound and 
$1-\tau_r(t,s) \ge 1/(8\pi^2)$ on the subset $\sd_{[-1,1]}(t,s)\le r/2$, proved in the Case 2 of the proof 
of Proposition \ref{prop:capX0}, to remove the term $ \big(1-\|x'\|^2 + 1-\tau_r(t,s) \big)^\mu$ from the 
integral. The rest of the proof then follows from that of the Case 2 of the proof of Proposition \ref{prop:capX0}.
This completes the proof. 
\end{proof}

\begin{cor}
For $d\ge 2$, $\b > -\frac{d+1}{2}$ and $\g > -\f12$, the space $(\XX^{d+1},  W_{\b,\g,\mu}^0, \sd_{\XX})$ is a 
homogeneous space. 
\end{cor}

When $\varrho =0$, $\b = 0$, $\mu =0$ and $\g = \f12$, the relation \eqref{eq:capX} is for the Lebesgue 
measure $\d \mb$ on the double cone, 
$$
     \mb \big(\cb_0((x,t), r)\big)  \sim r^{d+1}  (1-t^2+ r^2)^{\f12} (t^2 - \|x\|^2+n^{-2})^\f12.
$$
Furthermore, $W_{0,0,0}^0\big(\cb_0((x,t), r)\big) \sim r^{d+1}$ and $W_{0,0,0}^0(x,y) = (1-t^2)^{-\f12} 
(t^2-\|x\|^2)^{-\f12}$. 

\subsection{Orthogonal polynomials on the hyperboloid}

With respect to the weight function $W_{\b,\g,\mu}^\varrho$, we defined the inner product 
$$
  \la f,g\ra_{W} = \int_{\XX^{d+1}} f(x,t)g(x,t) W_{\b,\g,\mu}^\varrho(x,t) \d x \d t. 
$$
Let $\CV_n(\XX^{d+1}, W_{\b,\g,\mu}^\varrho)$ be the space of these orthogonal polynomials of degree $n$,
which has the dimension $\dim \CV_n(\XX^{d+1}, W) = \binom{n+d}{n}$. Like the decomposition on the surface 
$\XX_0^{d+1}$, this space satisfies 
$$
\CV_n\left(\XX^{d+1}, W_{\b,\g,\mu}^\varrho\right) =  \CV_n^E\big(\XX^{d+1}, W_{\b,\g,\mu}^\varrho\big) \bigoplus 
    \CV_n^O\big(\XX^{d+1}, W_{\b,\g,\mu}^\varrho\big),
$$
where the subspace $\CV_n^E(\XX^{d+1}, W_{\b,\g,\mu}^\varrho)$ consists of orthogonal polynomials that 
are even in $t$ variable, whereas the subspace $\CV_n^O(\XX^{d+1}, W_{\b,\g,\mu}^\varrho)$ consists of 
orthogonal polynomials that are odd in $t$ variable.  

An orthogonal basis can be given explicitly in terms of the Jacobi polynomials and classical orthogonal 
polynomials on the unit ball for the subspace $\CV_n^E(\XX^{d+1},W_{\b,\g,\mu}^\varrho)$ for all 
$\varrho \ge 0$, but for the subspace $\CV_n^O(\XX^{d+1},W_{\b,\g,\mu}^\varrho)$ only when $\varrho =0$. 
For example, let $\{P_\kb^{n-2k}: |\kb| = n-2k, \, \kb \in \NN_0^d\}$ denote an orthonormal basis of 
$\CV_{n-2k}(\BB^d,W_\mu)$. Then the polynomials 
\begin{align}\label{eq:solidOPhypG}
  \mathbf{C}_{n-2k,\kb}^n (x,t) = &\,  P_k^{(\g-\f12,n-2k+\b+\mu+\f{d-2}{2})}(2 t^2-2\varrho^2-1) \\
  &\times  (t^2-\varrho^2)^{\frac{n-2k}{2}} P_\kb^{n-2k}\bigg(\frac{x}{\sqrt{t^2-\varrho^2}} \bigg) \notag
\end{align}
with $|\kb| = n-2k$ and $0 \le k\le n/2$ form an orthogonal basis of $\CV_n^E(\VV^{d+1}, W_{\b,\g,\mu}^\varrho)$.
We call these polynomials generalized Gegenbauer polynomials on the solid hyperboloid. We will not 
work with the basis directly; see \cite{X20b} for the basis in other cases.

Let $\Pb_n^E(W_{\b,\g,\mu}^\varrho; \cdot,\cdot)$ denote the reproducing kernel of $\CV_n^E(\XX^{d+1},
W_{\b,\g,\mu}^\varrho)$. In terms of the basis of \eqref{eq:solidOPhypG}, we can write 
\begin{align*}
\Pb_n^E(W_{\b,\g,\mu}^\varrho; (x,t), (y,s)) = \sum_{m=0}^n \sum_{|\kb| = m}
   \frac{ \mathbf{C}_{n-2k,\kb}^n (x,t)  \mathbf{C}_{n-2k,\kb}^n (y,s)}{\la   \mathbf{C}_{n-2k,\kb}^n,
       \mathbf{C}_{n-2k,\kb}^n\ra_{W_{\b,\g,\mu}}}.
\end{align*}
Just like on the surface $\XX_0^{d+1}$, the kernel $\Pb_n^E \big (W_{\g,\mu}^\varrho; \cdot,\cdot \big)$ can
be used to study the Fourier orthogonal series of any function $f$ that is even in the variable $t$ on 
$\XX^{d+1}$. For such a function, $f(x,t) = f(x,-t)$, its projection on $\CV_n^O(\XX^{d+1},W_{\b,\g,\mu}^\varrho)$
becomes zero, so that its Fourier orthogonal expansion is given by 
\begin{equation} \label{eq:FourierX} 
  f = \sum_{n=0}^\infty \proj_n^E(W_{\b,\g,\mu}^\varrho; f),
\end{equation}
where the projection $\proj_n^E\big(W_{\b,\g,\mu}^\varrho\big): L^2\big(\XX^{d+1}, W_{\b,\g,\mu}^\varrho\big) \mapsto 
\CV_n^E\big(\XX^{d+1}, W_{\b,\g,\mu}^\varrho\big)$ can be written in terms of the kernel
$\Pb_n^E(W_{\b,\g,\mu}^\varrho; \cdot,\cdot)$ as 
$$
    \proj_n(W_{\b,\g,\mu}^\varrho ;f)= c_\sw \int_{\XX^{d+1}} 
       f(y) \Pb_n^E(W_{\b,\g,\mu}^\varrho; \cdot, (y,s))W_{\b,\g,\mu}^\varrho (s) \d y \d s. 
$$
Moreover, since $f$ is even in $t$ variable, we can regard it as the even extension of a function $f$ 
defined on the upper hyperboloid $\XX_{+}^{d+1}$, which is the cone $\VV^{d+1}$ when $\varrho =0$. 
In particular, this provides a Fourier orthogonal series for functions on the cone $\VV^{d+1}$, which is, 
however, different from the Fourier orthogonal series in the Jacobi polynomials on the cone discussed in
\cite{X21}.

The most interesting case on $\XX^{d+1}$ is $\b = \f12$. To simplify the notation, we shall denote 
$W_{\f12, \g,\mu}^\varrho$ by $W_{\g,\mu}^\varrho$ throughout the rest of the section; that is,
$$
   W_{\g,\mu}^\varrho(x,t): = |t| (1+\varrho^2-t^2)^{\g-\f12}(t^2-\varrho^2-\|x\|^2)^{\mu - \f12}.
$$
The orthogonal polynomials with respect to $W_{\g,\mu}$ also possess two characteristic properties:
the first one is the spectral operator that has orthogonal polynomials as eigenfunctions \cite[Theorem 4.8]{X20b}.

\begin{thm}\label{thm:solidHypGdiff}
Let $\varrho \ge 0$, $\g, \mu > -\f12$. Define the differential operator 
\begin{align*}
  \fD_{\g,\mu}^\varrho : = & (1+\varrho^2-t^2)  \left(1- \frac{\varrho^2}{t^2} \right) \partial_t^2 + 
  \Delta_x - \la x, \nabla_x\ra^2 + \la x, \nabla_x\ra \\
  & \,  +\frac{2}{t}(1+\varrho^2-t^2) \la x, \nabla_x\ra \partial_t  + 
       \left( (1+\varrho^2-t^2) \frac{\varrho^2}{t^2} + 2\mu+d \right)  \frac{1}{t}\partial_t  \notag \\
   & \, - (2\g+2\mu+d+1) \left( \left( 1- \frac{\varrho^2}{t^2}\right) t \partial_t + \la x,\nabla_x \ra \right)
\end{align*}
Then the polynomials in $\CV_n^E(\XX^{d+1}, W_{\g,\mu}^\varrho)$ are eigenfunctions of 
$ \fD_{\g,\mu}^\varrho$,
\begin{align} \label{eq:solidHypGdiff}
 \fD_{\g,\mu}^\varrho u = -n(n + 2 \g + 2 \mu+d) u, \qquad 
           \forall u \in  \CV_n^E\big(\XX^{d+1}, W_{\g,\mu}^\varrho\big)
\end{align}
\end{thm}

The second one is the addition formula for the reproducing kernel $\Pb_n^E\big(\sw_{\b,\g,\mu}^\varrho;\cdot,\cdot\big)$, 
which is of the simplest form when $\b = \f12$ \cite[Prop. 5.7 and Cor. 5.6]{X20b}. 
 
\begin{thm}
Let $d \ge 2$ and $\varrho \ge 0$. Then
\begin{enumerate}[  \quad \rm (a)]
\item For $\b, \g, \mu > -\f12$, 
\begin{align}\label{eq:PEhyp}
   \Pb_n^E \big(W_{\b,\g,\mu}^\varrho; (x,t), (y,s)\big) 
    = \Pb_n^E \left(W_{\b,\g,\mu}^0; \Big(x,\sqrt{t^2-\varrho^2}\Big),  \Big(y,\sqrt{s^2-\varrho^2} \Big)\right). 
\end{align} 
\item For  $\varrho =0$, and $\g, \mu \ge 0$, 
\begin{align} \label{eq:PEadd}
  \Pb_n^E \big (W_{\g,\mu}^0; (x,t),(y,s) \big) =    b_{\g,\mu} & \int_{-1}^1 \int_{-1}^1 
Z_n^{\g+\mu+\frac{d}{2}} \big(\xi(x,t,y,s; u,v)\big) \\ 
 &\times (1-v^2)^{\g-1} 
   (1-u^2)^{\mu-1} \d u \d v, \notag
\end{align}
where $b_{\g,\mu} = c_{\g-1,\g-1} c_{\mu-1,\mu-1}$ with $c_{a,b}$ defined as in \eqref{eq:c_ab} and
$$
  \xi(x,t,y,s; u,v) = \Big( \la x,y\ra+ u \sqrt{t^2-\|x\|^2}\sqrt{s^2-\|y\|^2}\Big) \mathrm{sign}(st) + v \sqrt{1-s^2}\sqrt{1-t^2},
$$
and the identity \eqref{eq:PEadd} holds under the limit when $\mu = 0$ or $\g = 0$.
\end{enumerate}
\end{thm}

\subsection{Highly localized kernels}
Let $\wh a$ be a cut-off function. For $(x,t)$, $(y,s) \in \XX^{d+1}$, the localized kernel 
$\Lb_n^E(W_{\g,\mu}^\varrho; \cdot,\cdot)$ is defined by
$$
   \Lb_n^E(W_{\g,\mu}^\varrho; (x,t),(y,s)) = \sum_{j=0}^\infty \wh a\left( \frac{j}{n} \right) 
         \Pb_j^E(W_{\g,\mu}^\varrho; (x,t), (y,s)). 
$$
We show that this kernel is highly localized when $(x,t)$ and $(y,s)$ are either both in $\XX_+^{d+1}$ 
or both in $\XX_-^{d+1}$. For $\mu, \g \ge 0$, define 
$$
     W_{\g,\mu}^\varrho (n; x,t) :=  \big(1+\varrho^2 - t^2+n^{-2}\big)^{\g}
          \big(t^2-\varrho^2 -\|x\|^2+n^{-2}\big)^{\mu}.
$$

\begin{thm} \label{thm:kernelXX}
Let $d\ge 2$, $\mu, \g \ge 0$. Let $\wh a$ be an admissible cutoff function. Then for
any $\k > 0$, and $(x,t)$, $(y,s)$ either both in $\XX_+^{d+1}$ or both in $\XX_-^{d+1}$,
\begin{equation*}
\left |\Lb_n^E (W_{\g,\mu}^\varrho; (x,t), (y,s))\right|
\le \frac{c_\k n^{d+1}}{\sqrt{W_{\g,\mu}^\varrho (n; x,t) }\sqrt{W_{\g,\mu}^\varrho (n; y,s) }}
    \big(1 + n \sd_{\XX}( (x,t), (y,s)) \big)^{-\k}.
\end{equation*}
\end{thm}

\begin{proof}
Again, it is sufficient to consider the case $\varrho =0$ and we shall be brief. By \eqref{eq:PEadd}
we can write $\Lb_n^E\big(W_{\g,\mu}^0\big)$ in terms of the kernel for the Jacobi polynomials. Let 
$\l = \g+\mu+\frac{d}{2}$. Then
\begin{align*}
\Lb_n^E \big(W_{\g,\mu}^0; (x,t), (y,s) \big) =   b_{\g,\mu} & \int_{-1}^1 \int_{-1}^1  L_n ^{(\l-\f12,\l-\f12)} 
       \big(\xi(x,t,y,s;u, v) \big) \\  
       & \qquad \times (1-v^2)^{\g-1} (1-u^2)^{\mu-1} \d u \d v.
\end{align*}
Hence, applying \eqref{eq:DLn(t,1)} with $m=1$ and $\alpha = \b =\l-1/2$, we obtain
\begin{align*}
 \left|\Lb_n^E (W_{\g,\mu}^0; (x,t), (y,s) ) \right| \le c n^{2 \l +1} \int_{-1}^1 \int_{-1}^1
     &  \frac{1}{ \left(1+ n \sqrt{1- \xi(x,t,y,s; u, v)}\right)^{\k+3\g+3\mu+2}}\\
       &   \times  (1-v^2)^{\g-1} (1-u^2)^{\mu-1} \d u \d v.
\end{align*}
Since $t$ and $s$ have the same sign, it is easy to verify that  
\begin{align*} 
  1- \xi(x,t,y,s;u,v) \, &  = 1- \cos \sd_{\XX} ((x,t), (y,s))\\ 
     & + (1-u) \sqrt{t^2-\|x\|^2}\sqrt{s^2-\|y\|^2}+  (1-v) \sqrt{1-t^2}\sqrt{1-s^2}. \notag
\end{align*}
The entries in both two lines in the right-hand side of the above identity are lower bounds of $1-\xi(x,t,u,s;u,v)$. 
Using the first one, we obtain the estimate
\begin{align*}
  \left| \Lb_n^E (W_{\g,\mu}^0; (x,t), (y,s) )\right| \le c n^{2 \l +1}  
   \frac{1}{(1+  n \sd_{\XX}((x,t),(y,s)))^{\k+\g+\mu}} I (x,t,y,s),
\end{align*}
where, using the second one and the symmetry of the integral, the integral $I(x,t,y,s)$ is given by 
$$
I (x,t,y,s)=c_{\g-\f12,\mu-\f12} \int_{0}^1 \int_0^1
                  \frac{(1-u^2)^{\mu-1}(1-v^2)^{\g-1}} {\left(1+n\sqrt{ A(1-u) + B (1-v)}\right)^{2\g+2\mu+2}} \d u \d v
$$
with $A = \sqrt{t^2-\|x\|^2}\sqrt{s^2-\|y\|^2}$ and $B =  \sqrt{1-t^2}\sqrt{1-s^2}$. The integral $I(x,t,y,s)$
is bounded by 1 and it can also be bounded by applying \eqref{eq:B+At} twice. Carrying out the 
estimates, we conclude that 
\begin{align*}
  I (x,t,y,s) & \le  c \frac{n^{- 2 \g - 2 \mu}}{ \left(A + n^{-1}\right)^\mu \left(B + n^{-1}\right)^\g} \\
    &  \le c  \frac{n^{- 2 \g-2\mu }}{\sqrt{W_{\g,\mu}^0 (n; t)} \sqrt{W_{\g,\mu}^0 (n; t)}}
     \big(1+  n \mathsf{d}_{\XX}((x,t),(y,s))\big)^{\g+\mu},  
\end{align*}
where the second inequality follows from \eqref{eq:ab+=a+b+} and Lemma \ref{lem:|s-t|XX}. Putting the last
two displayed inequalities together, we have established (ii).  
\end{proof}
 
This establishes Assertion 1. The next theorem establishes Assertion 2. 

\begin{thm} \label{thm:L-LkernelXX}
Let $d\ge 2$, $\mu, \g \ge -\f12$. Then for $(x_i,t_i)$ and $(y,s)$ that are either all in $\XX_+^{d+1}$ or 
all in $\XX_-^{d+1}$, $(x_1,t_1) \in \cb_\varrho((x_2,t_2), c_*n^{-1})$ with $c_*$ small and any $\k >0$,
\begin{align*}
& \left|\Lb_n^E (W_{\g,\mu}^\varrho;(x_1,t_1), (y,s))-\Lb_n^E (W_{\g,\mu}^\varrho; (x_2,t_2),(y,s))\right| \\
& \quad \qquad \le \frac{c_\k n^{d+1}  \sd_{\XX}^{\varrho}((x_1,t_1),(x_2,t_2))}
    {\sqrt{W_{\g,\mu}^\varrho(n,t_2)}\sqrt{W_{\g,\mu}^\varrho (n; s)} \big(1+n \sd_{\XX}^{\varrho}((x_2,t_2),(y,s))\big)^\k }.  
\end{align*} 
\end{thm}

\begin{proof}
Again it suffices to prove the case $\varrho =0$. Let $\xi_i (u,v) = \xi(x_i,t_i,y,s;u,v)$. From
\eqref{eq:d=d+dXX2}, it is easy to see that 
\begin{align*}
 \xi_1 (u,v) - \xi_2(u,v)  = \,& \cos \sd_{\XX}((x_1,t_1),(y,s)) - \cos \sd_{\XX}((x_2,t_2),(y,s)) \\
              &  + (1-u) \left(\sqrt{t_2^2-\|x_2\|^2} - \sqrt{t_1^2-\|x_1\|^2}\right) \sqrt{s^2-\|y\|^2} \\
              &  + (1-v) \left(\sqrt{1-t_2^2} - \sqrt{1-t_1^2}\right) \sqrt{1-s^2}.
\end{align*}
By Lemma \ref{lem:|s-t|XX} and the proof of Theorem \ref{thm:L-LkernelX0}, this leads to
$$
   |\xi_1(u,v) - \xi_2(u,v)| \le \sd_{\XX}\big((x_1,t_1), (x_2,t_2)\big) \big[ \Sigma_1 + \Sigma_2 (u)+ \Sigma_3 (v) \big], 
$$
where 
\begin{align*}
   \Sigma_1 \,& =  \sd_{\XX}((x_2,t_2), (y,s)) +\sd_{\XX}((x_1,t_1), (x_2,t_2)), \\
   \Sigma_2(u) \,& = (1-u) \sqrt{s^2-\|y\|^2}, \\
   \Sigma_3(v) \,& = (1-v) \sqrt{1-s^2}.
\end{align*}
Hence, following the proof of Theorem \ref{thm:L-LkernelX0}, we see that, with $\l = \g+\mu+ \frac{d}{2}$,
\begin{align*}
& \left |\Lb_n (W_{\g,\mu}^0; (x_1,t_1), (y,s))-\Lb_n (W_{\g,\mu}^0; (x_2,t_2), (y, s))\right| \\
 &    \le c \, \sd_{\XX}\big((x_1,t_1),(x_2,t_2)\big) \int_{-1}^1  \int_{-1}^1 \left[  \frac{n^{2 \l + 3}}
    {\big(1+n\sqrt{1-\xi_1(u,v)^2} \big)^{\k}} +  \frac{n^{2 \l + 3}}{\big(1+n\sqrt{1-\xi_2(u,v)^2}\big)^{\k}} \right ] \\
 & \qquad\qquad\qquad\qquad\qquad\qquad\qquad \times \big(\Sigma_1+\Sigma_2(u)+\Sigma_3(v)\big) (1-u^2)^{\mu-1}  (1-v^2)^{\g-1} \d u\d v.
\end{align*}
The integrals that contain $\Sigma_1$ and $\Sigma_3(v)$ can be estimated exactly as in the proof of 
Theorem \ref{thm:L-LkernelX0}. The integral that contains $\Sigma_2(u)$ does not cause additional 
problem and can be handled just as the integral containing $\Sigma_3(v)$. We omit the details. 
\end{proof}

The case $p=1$ of the following lemma establishes Assertion 3 for $W_{\g,\mu}^\varrho$. 

\begin{lem}\label{lem:intLnX}
Let $d\ge 2$, $\g > - \f12$ and $\mu > -\f12$. For $0 < p < \infty$, assume 
$\k > \frac{2d+2}{p} + 2(\b+\g) |\f1p-\f12|$. Then for $(x,t) \in \XX^{d+1}$,  
\begin{align*}
\int_{\XX^{d+1}} \frac{W_{\g,\mu}^\varrho(y,s)}{  W_{\g,\mu}^\varrho (n; y,s)^{\f{p}2}
    \big(1 + n \sd_{\XX}^\varrho( (x,t), (y,s)) \big)^{\k p}}   \d y \d s 
    \le c n^{-d-1} W_{\g,\mu}^\varrho (n; x,t)^{1-\f{p}{2}}.
\end{align*}
\end{lem}

\begin{proof}
Again, it suffices to consider $\varrho = 0$. Let $J_p$ denote the left-hand side the inequality to be proved. 
As in the case of $\XX_0$, we only need to estimate $J_2$. Furthermore, following the proof of 
Lemma \ref{lem:intLnX0}, we only need to estimate the integral in $J_2$ over either $\XX_+^{d+1}$ or 
$\XX_-^{d+1}$, which we choose as $\XX_+^{d+1} = \VV^{d+1}$ and denote it by $J_{2,+}$. Then 
\begin{align*}
J_{2,+}  = \int_0^1 s^d \int_{\BB^d} \frac{W_{\g,\mu}^0(sy',s)  } {W_{\g,\mu}^0 (n;sy', s) 
      (1 + n \sd_{\XX}( (x,t), (sy',s)) )^{2\k}}  \d  y' \d s.
\end{align*} 
Let $x = t x'$ and $y = s y'$ with $x', y' \in \BB^d$. Using $\sw_{0,\g}^0(t) = (1-t^2)^{\g-\f12}$ and
$\sw^0_{\b,\g}(n; s) = (1-t^2+n^{-2})^{\g}$, we can easily verify that 
$$
 \frac{W_{\g,\mu}^0(y,s) } {W_{\g,\mu}^0 (n;y, s)}  
     \le c \frac{\sw_{0,\g}^0(s)}{\sqrt{1-\|y'\|^2} \sw^0_{0,\g}(n; s)},
$$
which leads to 
\begin{align*}
J_{2,+}  \le \int_0^1 s^d \int_{\BB^d} \frac{\sw_{0,\g}^0(s) } { \sw^0_{0,\g}(n; s)
      (1 + n \sd_{\XX}( (x,t), (sy',s)) )^{2\k}}  \frac{\d  y' }{\sqrt{1-\|y'\|^2}}\d s.
\end{align*} 
Setting $x= t x'$, $X = (x', \sqrt{1-\|x'\|^2})$ and $Y = (y', \sqrt{1-\|y'\|^2})$, so that $\sd_{\BB^d}(x',y') 
= \sd_{\SS^d}(X,Y)$, we use the identity 
\begin{equation}\label{eq:BB-SS+}
   \int_{\BB^d} g\left(y', \sqrt{1-\|y'\|^2} \right) \frac{d y'}{\sqrt{1-\|y'\|^2}} =  \int_{\SS_+^d} g(y) \d \s(y),
\end{equation}
where $\SS_+^d$ denotes the upper hemisphere of $\SS^d$, which allows us to follow the proof of
Lemma \ref{lem:intLnX0} to obtain 
\begin{align*}
 J_{2,+}  \, &\le c  \int_0^1 \int_{-1}^1 \frac{ s^{d} \sw_{0,\g}^0 (s)  (1-u^2)^{\f{d-2}{2}} }{\sw^0_{0,\g}(n; s)
        \left(1 + n \sqrt{1-  t s u - \sqrt{1-t^2}\sqrt{1-s^2}}\right)^{2\k} } \d u  \d s.
\end{align*}
The integral in the right-hand side with $d$ replaced by $d+1$ appeared in the proof of Lemma \ref{lem:intLnX0}, 
and it is bounded by $c n^{-d-1}$ accordingly. 
\end{proof}
 
\begin{prop}\label{prop:intLnXX}
For $\g \ge 0$, $\mu \ge 0$ and $(x,t) \in \XX^{d+1}$,  
\begin{equation*}
   \int_{\XX^{d+1}} \left| \Lb_n^E \big(W_{\g,\mu}^\varrho; (x,t), (y,s)\big) \right|^p W_{\g,\mu}^\varrho(y,s) 
       \d y \d s \le \bigg (\frac{n^d}{W_{\g,\mu}^\varrho(n;t)} \bigg)^{p-1}. 
\end{equation*}
\end{prop}

This follows by applying Lemma \ref{lem:intLnX} on the estimate in Theorem \ref{thm:kernelXX}.

We have established Assertions 1 -- 3 for $\Lb_n^E( W_{\g,\mu};\cdot,\cdot)$. The kernel uses, 
however, only polynomials that are even in $t$ and in $s$ variable. 

\begin{cor}
For $d\ge 2$, $\varrho \ge 0$, $\g \ge 0$ and $\mu \ge 0$, the space $(\XX^{d+1}, W_{\g,\mu}^\varrho, 
\sd_{\XX}^\varrho)$ is a localizable homogeneous space, where its localized kernels are defined for
polynomials even in $t$ and in $s$ variables. 
 \end{cor}

\subsection{Maximal $\ve$-separated sets on the hyperbolic surface} \label{sec:ptsX}

We give a construction of maximal $\ve$-separated on the hyperboloid and the double cone, following
the definition of Definition \ref{defn:separated-pts}.   
Our construction follows the one on $\XX_0^{d+1}$ in Section \ref{sec:ptsX0}. We first need $\ve$-separated 
set on the unit ball $\BB^d$. We adopt the following notation. For $\ve > 0$, we denote by $\Xi_{\BB}(\ve)$ a 
maximal $\ve$-separated set on the unit ball $\BB^d$ and we let $\BB_u(\ve)$ be the subsets in $\BB^d$ 
so that the collection $\{\BB_u(\ve): u \in \Xi_\BB(\ve)\}$ is a partition of $\BB^d$, and we assume
\begin{equation}\label{eq:ptsB1}
  \cb_{\BB}(u, c_1 \ve) \subset \BB_u(\ve) \subset \cb_{\BB}(u, c_2 \ve), \qquad u \in \Xi_{\BB}(\ve),
\end{equation} 
where $\cb_{\BB}(u,\ve)$ denotes the ball centered at $u$ with radius $\ve$ in $\BB^d$, $c_1$ and 
$c_2$ depend only on $d$. It is known (see, for example, \cite{PX2}) that such a $\Xi_\BB(\ve)$ 
exists for all $\ve > 0$ and its cardinality satisfies  
\begin{equation}\label{eq:ptsB2}
c_d' \ve^{-d} \le \# \Xi_{\BB}(\ve) \le c_d \ve^{-d}.
\end{equation} 
For the hyperboloid $\XX^{d+1}$, we denote by $\Xi_{\XX} = \Xi_{\XX}(\ve)$ a maximum $\ve$-separated 
set and, furthermore, denote by $\{\XX(u,t): (tu,t) \in \Xi_{\XX}\}$ a partition of $\XX^{d+1}$. 
We give one construction of such sets below.

Let $\ve > 0$ and let $N = 2 \lfloor \frac{\pi}{2}\ve^{-1} \rfloor$. We define $t_j = \cos \t_j$ and 
$t_j^+$ and $t_j^-$, $1 \le j \le N$, as in Subsection \ref{sec:ptsX0}. Then $\XX^{d+1}$ can be 
partitioned by 
$$
    \XX^{(j)}:=  \left\{(x,t) \in \XX^{d+1}:   t_j^- < t \le t_j^+ \right \}, \qquad 1 \le j \le N.
$$
Furthermore, the upper and lower hyperboloid $ \XX_{+}^{d+1} $ and  $\XX_{-}^{d+1}$ can be partitioned by 
$$
 \XX_{+}^{d+1} =\bigcup_{j=1}^{N/2} \XX^{(j)}  \quad\hbox{and}\quad \XX_{-}^{d+1} =\bigcup_{j= N/2 +1}^N\XX^{(j)}.
$$
Let $\ve_j := \pi \ve/(2 t_j)$. Then $\Xi_\BB(\ve_j)$ is the maximal $\ve_j$-separated set of $\BB^d$ 
such that, for each $j \ge 1$, $\{\BB_u(\ve_j):  u \in \Xi_\BB(\ve_j)\}$ is a partition of $\BB^d$ and 
$
   \# \Xi_\BB(\ve_j) \sim \ve_j^{-d}.
$
For each $j =1,\ldots, N$, we decompose $\XX^{(j)}$ by 
$$
 \XX^{(j)} =  \bigcup_{u \in  \Xi_\BB(\ve_j)} \XX(u,t_j), \quad \hbox{where}\quad 
 \XX(u,t_j):= \left\{(t v,t):  t_j^- < t \le t_j^+, \, v \in \BB_u(\ve_j) \right\}.
$$
Finally, we define the set $\Xi_{\XX}$ of $\XX^{d+1}$ by
$$
   \Xi_{\XX} = \big\{(t_j u, t_j): \,  u \in \Xi_\BB(\ve_j), \, 1\le j \le N \big\}. 
$$

\begin{prop} \label{prop:subsetX}
Let $\ve > 0$ and $N = 2 \lfloor \frac{\pi}{2} \ve^{-1} \rfloor$. Then $\Xi_{\XX}$ is a maximal $\ve$-separated 
set of $\XX^{d+1}$ and $\{\XX(t_j u, t_j): u \in \Xi_\BB(\ve_j), \, 1\le j \le N \}$ is a partition 
$$
   \XX^{d+1} =  \bigcup_{(t u,t) \in \Xi_\XX} \XX(u,t)= 
         \bigcup_{j=1}^N \bigcup_{u \in \Xi_\BB(\ve_j)} \XX(u,t_j).     
$$
Moreover, there are positive constants $c_1$ and $c_2$ depending only on $d$ such that 
\begin{equation}\label{eq:incluXcap}
      \cb_0 \big((t_j u,t_j), c_1 \ve\big) \subset \XX(u,t_j) \subset \cb_0 \big( (t_j u,t_j), c_2 \ve\big), \quad (t_j u, t_j) \in
       \Xi_\XX, 
\end{equation}
and $c_d'$ and $c_d$  depending only on $d$ such that 
$$
      c_d' \ve^{-d-1} \le \# \Xi_{\XX} \le c_d \ve^{-d-1}. 
$$
\end{prop}

\begin{proof}
The proof is parallel to that of Propositions \ref{prop:subsetX0} and follows almost verbatim.
We omit the details.  
\end{proof}

\begin{prop}
For $\ve > 0$, let $\Xi_{\XX}$ be a maximal $\ve$-separated set in $\XX^{d+1}$. Define 
$$
\Xi_{\XX}^\varrho: = \left\{\left(x,\sqrt{t^2-\varrho^2}\right): (x,t) \in \Xi_{\XX}\right \}.
$$
Then $\Xi_{\XX}^\varrho$ is a maximal $\ve$-separated set in the solid hyperboloid ${}_\varrho\XX^{d+1}$.
\end{prop}

This is an immediate consequence of \eqref{eq:distXX0rho}. In particular, for $\Xi_{\XX}$
defined in Proposition \ref{prop:subsetX}, both \eqref{eq:incluX0cap} and \eqref{eq:incluX0cap2} 
extend to ${}_\varrho\XX^{d+1}$ as well. Moreover, this set is also evenly symmetric, where the
notion of evenly symmetric set on $\XX^{d+1}$ is defined analogously as in Definition 
\ref{def:evensymmetry} with $\XX_0$ replaced by $\XX$. 

With Assertions 1--3 established for $W_{\g,\mu}$ for $\g \ge 0$ and $\mu \ge 0$, we can now 
apply \cite[Theorem 2.15]{X21} to state the Marcinkiewicz-Zygmund inequality. However,
following the consideration in the case of $\XX_0^{d+1}$, our result holds under symmetry assumption
on the weight $W$ and the set $\Xi_\XX$, and is restricted to the subspace of polynomials 
$$
    \Pi_n^E(\XX^{d+1}) = \left\{p \in \Pi_n^{d+1}: p(x,t) = p(x,-t), \forall (x,t)\in \XX^{d+1}\right\}.
$$
 
\begin{thm} \label{thm:MZinequalityX}
Let $W$ be a doubling weight on $\XX^{d+1}$ such that $W(x,t) = W(x,-t)$ for all $(x,t)\in\RR^{d+1}$.
Let  $\Xi_{\XX}^\varrho$ be a symmetric maximal $\f \delta n$-separated subset of $\XX^{d+1}$ 
and $0 < \delta \le 1$.
\begin{enumerate}[$(i)$]
\item For all $0<p< \infty$ and $f\in\Pi_m^E(\XX^{d+1})$ with $n \le m \le c n$,
\begin{equation*}
  \sum_{z \in \Xi_{\XX}^\varrho} \bigg( \max_{(x,t)\in \cb_\varrho((z,r), \f \delta n)} |f(x,t)|^p \bigg)
     W\!\left(\cb_\varrho((z, r), \tfrac \delta n) \right) \leq c_{W,p} \|f\|_{p,W}^p
\end{equation*}
where $c_{W,p}$ depends  on $p$ when $p$ is close to $0$ and on the doubling constant of $W$.
\item For $0 < r < 1$, there is a $\delta_r > 0$ such that for $\delta \le \delta_r$, $r \le p < \infty$ and 
$f \in \Pi_n^E(\XX^{d+1})$,  
\begin{align*}
  \|f\|_{p,W}^p \le c_{W,r} \sum_{z \in\Xi^\varrho}
       \bigg(\min_{(x,t)\in \cb_\varrho\bigl((z,r), \tfrac{\delta}n\bigr)} |f(x,t)|^p\bigg)
          W\bigl(\cb_\varrho((z,r), \tfrac \delta n)\bigr)
\end{align*}
where $c_{W,r}$ depends only on the doubling constant of $W$ and on $r$ when $r$ is close to $0$.
\end{enumerate}
\end{thm}

\subsection{Positive cubature rules} \label{sec:CFframeX}
We need fast decaying polynomials on the solid hyperboloid, which will verify Assertion 4.

\begin{lem}\label{lem:A4X}
Let $d\ge 2$ and $\varrho \ge 0$. For each $(x,t) \in \XX^{d+1}$, there is a polynomial $\CT_{x,t}^\varrho$ 
in $\Pi_n^E$ that satisfies
\begin{enumerate} [   (1)]
\item $\CT_{x,t}^\varrho(x,t) =1$, $\CT_{x,t}^\varrho(y,s) \ge c > 0$ if $(y,s) \in \cb( (x,t), \f{\delta}{n})$, 
and for every $\k > 0$,
$$
   0 \le \CT_{x,t}^\varrho(y,s) \le c_\k \left(1+ \sd_{\XX}^\varrho \big((x,t),(y,s)\big) \right)^{-\k}, \quad (y,s) \in \XX^{d+1}.
$$
\item there is a polynomial $q$ of degree $4n$ such that $q(x,t) \CT_{x,t}^\varrho$ is a polynomial of degree $5 n$ 
in $(x,t)$ and $1 \le q(x,t) \le c$. 
\end{enumerate}
\end{lem}

\begin{proof}
We construct our polynomial based on the one given in Lemma \ref{lem:A4X0} by following the 
approach in Lemma \ref{lem:A4X0}. For $(x,t), (y,s) \in \XX^{d+1}$, we introduce the notation 
$X = (x, \sqrt{t^2- \varrho^2-\|x\|^2})$ and 
$Y = (y, \sqrt{s^2- \varrho^2-\|y\|^2})$.  Moreover, denote $X_* =  (x, - \sqrt{t^2- \varrho^2-\|x\|^2})$ and 
$Y_* =  (y, - \sqrt{s^2- \varrho^2-\|y\|^2})$. Then $(X,t), (Y,s)$ and $(X_*,t)$ and $(Y_*, s)$ are all elements 
of $\XX_0^{d+2}$. Let $T_{(X,t)}^\varrho$ denote the polynomial 
of degree $n$ on $\XX_0^{d+2}$ defined in Lemma \ref{lem:A4X0}. We now define
$$
   \CT^\varrho_{x,t}(y,s): = \frac{T_{(X,t)}^\varrho (Y,s)+ T_{(X,t)}^\varrho (Y_* ,s)} {1 + T_{(X,t)}^\varrho(X_*,t)}. 
$$
Since $T_{(X,t)}^\varrho(X,t) =1$, it follows that $\CT_{(x,t)}^\varrho(x,t) = 1$. Moreover, since  
$$
T_{(X,t)}^\varrho (X_*,t) = \frac{S_n (2\|x\|^2 + 1 + 2\varrho^2 -2 t^2) + S_n (2\|x\|^2-1)}{1 + S_n(2 t^2-2\varrho^2-1)}
$$
and $0 \le S_n(t) \le c$, we see that $0 \le T_{(X,t)}(X_*,s) \le c$. In particular, it follows that
$$
    \CT_{x,t}^\varrho(y,s) \ge c \,T_{(X,t)}^\varrho (Y,s) \ge c > 0, \qquad (y,s) \in \cb\big((x,t), \tfrac \delta n\big), 
$$
since $\sd_{\XX}((x,t),(y,s)) = \sd_{\XX_0^{d+2}}((X,t),(Y,s))$. Furthermore, since 
$\cos \sd_{\XX_0} ((X,t), (Y,s)) \ge \cos \sd_{\XX_0} ((X,t), (Y_*,s))$, we obtain 
$$
\sd_{\XX_0} ((X,t), (Y_*,s)) \ge  \sd_{\XX_0} ((X,t), (Y,s)) = \sd_{\XX}((x,t),(y,s)). 
$$
Hence, using the estimate of $T_{(X,t)}^\varrho$ in Lemma \ref{lem:A4X0}, we conclude that
\begin{align*}
 0 \le  \CT_{x,t}^\varrho(y,s) \,& \le c\left[\big(1+n \sd_{\VV_0}((X,t),(Y,s) ) \big)^{-\k} 
      + \big(1+n \sd_{\VV_0}((X,t),(Y_*,s)) \big)^{-\k}\right] \\
      \,& \le c \big(1+n \sd_{\VV}((x,t),(y,s)) \big)^{-\k}. 
\end{align*} 
Finally, let $q(x,t) = (1+S_n(2t^2-2\varrho^2 -1)) T_{(X,t)}^\varrho(X_*,t)$. Then $q(x,t)$ is a polynomial of degree 
at most $4n$, so that $q(x,t) \CT_{x,t}^\varrho$ is a polynomial of degree at most $5 n$ and $1 \le q(x,t) \le c$. 
This completes the proof.
\end{proof}

The lemma establishes Assertion 4 with a polynomial in $\Pi_n^E(\XX^{d+1})$. Let $W$ be a doubling 
weight function that is even in $t$ variable on $\XX^{d+1}$. We define 
\begin{align}\label{eq:ChristoffelFX}
   \l_n^E(W;x,t): = \inf_{\substack{g(x,t) =1 \\ g \in \Pi_n^E(\XX^{d+1})}} \int_{\XX^{d+1}} 
         |g(x,t)|^2 W(x,t)  \d x \d t, 
\end{align} 
which is the Christoffel function for the space $\Pi_n^E(\XX^{d+1})$. 
Just like the case of $\XX_0^{d+1}$, this Christoffel function is related to the kernel $\Kb_n^E(W) = 
\sum_{k=0}^n\Pb_k^E(W)$ by 
\begin{align*} 
   \l_n^E(W; x,t) = \frac{1}{\Kb_n^E(W; (x,t),(x,t))}, \qquad (x,t) \in \XX^{d+1}.
\end{align*}    
Moreover, \cite[Propositions 2.17 and 2.18]{X21} remain valid for $\l_n^E(W)$ on $\XX^{d+1}$ by 
using the Lemma \ref{lem:A4X}. Hence, we obtain the following corollary. 

\begin{cor} \label{cor:ChristFX}
Let $W$ be a doubling weight function on $\XX^{d+1}$ such that $W(x,t)= W(x,-t)$ for all
$(x,t) \in \XX^{d+1}$. Then 
\begin{equation*}  
   \l_n^E \big(W; (x,t) \big)  \le c \, W\!\left(\cb_\varrho\left((x,t), \tfrac1n \right) \right).
\end{equation*}
Moreover, for $\g \ge 0$ and $\mu \ge 0$, 
$$
\l_n^E \big(W_{\g,\mu}; (x,t) \big)  \ge c \, W_{\g,\mu}\!\left(\cb_\varrho\left((x,t), \tfrac1n \right) \right) 
   = c n^{-d} W_{\g,\mu}^\varrho(n; x,t).
$$
\end{cor}

We are now in position to state the positive cubature rule for solid hyperboloid, following 
which holds  for polynomials in $\Pi_n^E(\XX^{d+1})$ and under the symmetry assumptions.

\begin{thm}\label{thm:cubatureX}
Let $d \ge 2$ and $\varrho \ge 0$. Let $W$ be a doubling weight on $\XX^{d+1}$ such that 
$W(x,t)= W(x,-t)$ for all $(x,t) \in \XX^{d+1}$. Let $\Xi^\varrho$ be a symmetric maximum 
$\frac{\delta}{n}$-separated subset of $\XX^{d+1}$. There is a $\delta_0 > 0$ such that for 
$0 < \delta < \delta_0$ there exist positive numbers $\l_{z,r}$, $(z,r) \in \Xi^\varrho$, so that 
\begin{equation}\label{eq:CFX}
    \int_{\XX^{d+1}} f(x,t) W(x,t)  \d x \d t  = \sum_{(z,r) \in \Xi^\varrho}\l_{z,r} f(z,r), \qquad 
            \forall f \in \Pi_n^E(\XX^{d+1}). 
\end{equation}
Moreover, $\l_{z,r} \sim W\!\left(\cb_\varrho((z,r), \tfrac{\delta}{n})\right)$ for all $(z,r) \in \Xi^\varrho$. 
\end{thm}

Again, this follows from \cite[Theorem 2.20]{X21} when the domain becomes $\XX_0^{d+1}$ and it 
remains valid under symmetry assumptions on the weight and on the polynomials.

\subsection{Localized tight frames}
The symmetry assumption for the weight and for the $\ve$-separated subset also carries over to the 
local frame on the solid hyperboloid. Let $\varrho \ge 0$ and let $W$ be a doubling weight on $\XX^{d+1}$ 
that is even in $t$ variable. Let $\Lb_n^E(W)*f$ denote the operator in $\Pi_{2n}^E(\XX^{d+1})$ defined by 
$$
   \Lb_n^E(W) * f (x): = \int_{\XX^{d+1}}f(y,s) \Lb_n^E(W; (x,t),(y,s)) W(y,s)  \d y \d s,
$$
where $\Lb_n^E(W; \cdot,\cdot)$ is the highly localized kernel defined via a cut-off function $\wh a$ that
satisfies \eqref{eq:a-frame}. For $j =0,1,\ldots,$ let $\Xi_j^\varrho$ be a symmetric maximal 
$\frac \delta {2^{j}}$-separated subset in $\XX^{d+1}$, so that 
\begin{equation*}
  \int_{\XX^{d+1}} f(x,t) W(x,t)  \d x \d t = \sum_{(z,r) \in \Xi_j^\varrho} \l_{(x,r),j} f(z,r), 
     \quad f \in \Pi_{2^j}^E(\XX^{d+1}).  
\end{equation*}
For $j=1,2, \ldots,$ define the operator $F_j^\varrho(W)$ by
$$
  F_j^\varrho(W) * f = \Lb_{2^{j-1}}^E(W) * f
$$
and define the frame elements $\psi_{(z,r),j}$ for $(z,r) \in \Xi_j^\varrho$ by 
$$ 
   \psi_{(z,r),j}(x,t):= \sqrt{\l_{(z,r),j}} F_j^\varrho((x,t), (z,r)), \qquad (x,t) \in \XX^{d+1}. 
$$ 
Then $\Phi =\{ \psi_{(z,r),j}: (z,r) \in \Xi_j^\varrho, \, j =1,2,3,\ldots\}$ is a tight frame by
\cite[Theorem 2.21]{X21}. 

\begin{thm}\label{thm:frameX}
Let $W$ be a doubling weight on $\XX^{d+1}$ even in its $t$ variable. If 
$f\in L^2(\XX^{d+1}, W)$ and $f$ is even in $t$ variable, then 
$$ 
   f =\sum_{j=0}^\infty \sum_{(z,r) \in\Xi_j^\varrho}
            \langle f, \psi_{(z,r), j} \rangle_\sw \psi_{(z,r),j}  \quad\mbox{in $L^2(\XX^{d+1}, W)$}
$$  
and
$$ 
\|f\|_{2, W}  = \Big(\sum_{j=0}^\infty \sum_{(z,r) \in \Xi_j^\varrho} 
       |\langle f, \psi_{(z,r),j} \rangle_W|^2\Big)^{1/2}.
$$ 
Furthermore, for $\g \ge 0$ and $\mu \ge 0$, the frame for $W_{\g,\mu}^\varrho$ is highly localized in the 
sense that, for every $\k >0$, there exists a constant $c_\k >0$ such that 
\begin{equation} \label{eq.needleX}
   |\psi_{(z,r),j}(x,t)| \le c_\k \frac{2^{j (d+1)/2}}{\sqrt{ W_{\g,\mu}^\varrho(2^{j}; x,t)} 
       (1+ 2^j \sd_{\XX}^\varrho((x,t),(z,r)))^\k},      \quad (x,t)\in \XX^{d+1}.
\end{equation}
\end{thm}
 
Like the case on the hyperbolic surface, the frame elements involve only orthogonal polynomials even in 
$t$ variable. The localization \eqref{eq.needleX} follows from Theorem \ref{thm:kernelXX} and 
$\l_{(z,r),j} \sim 2^{- j(d+1)} W_{\g,\mu}^\varrho(2^j;x,t)$ which holds for $W_{\g,\mu}^\varrho$ by 
Corollary \ref{cor:ChristFX} and \eqref{eq:capX}. We note, however, that the localized tight frame holds
for the weight function $W_{\g,\mu}^\varrho$, which does not include the Lebesgue measure on 
$\XX^{d+1}$.

\subsection{Characterization of best appoximation}

For $f\in L^p(\XX^{d+1}, W_{\g,\mu}^\varrho)$, we denote by $\Eb_n(f)_{p, W_{\g,\mu}^\varrho}$ the 
best approximation to $f$ from $\Pi_n^{d+1}(\XX^{d+1})$ in the norm $\|\cdot\|_{p, W_{\g,\mu}^\varrho}$; that is, 
$$
    \Eb_n(f)_{p, W_{\g,\mu}^\varrho}:= \inf_{g \in \Pi_n(\XX^{d+1})} \|f - g\|_{p, W_{\g,\mu}^\varrho},
        \qquad 1 \le p \le \infty.
$$
As in the case of hyperbolic surface, if $f$ is even in $t$ variable, then we can choose the polynomial of 
best approximation from $\Pi_n^E(\XX^{d+1})$. We can give a characterization of this quantity for 
functions that are even in $t$ variable. 

For $f\in L^p(\XX^{d+1}, W_{\g,\mu}^\varrho)$ and $r > 0$, the modulus of smoothness is defined by
$$
  \o_r(f; \rho)_{p,W_{\g,\mu}^\varrho} = \sup_{0 \le \t \le \rho} 
     \left\| \left(I - \Sb_{\t,W_{\g,\mu}^\varrho}\right)^{r/2} f\right\|_{p,W_{\g,\mu}^\varrho}, \quad 1 \le p \le \infty, 
$$
where the operator $\Sb_{\t,W_{\g,\mu}^\varrho}$ is defined by, for $n = 0,1,2,\ldots$ and $\l = \g+\mu+\f{d}2$, 
$$
 \proj_n^E\Big(W_{\g,\mu}^\varrho; \Sb_{\t,W_{\g,\mu}^\varrho}f\Big) = R_n^{(\l-\f12, \l-\f12)} (\cos \t) 
      \proj_n^E\Big(W_{\g,\mu}^\varrho; f\Big).
$$
In terms of the fractional differential operator $(-\fD_{\g,\mu}^\varrho)^{\f r 2}$, the $K$-functional for a doubling 
weight $W$, even in $t$ variable on $\XX^{d+1}$, is defined by
$$
   \Kb_r(f,t)_{p,W} : = \inf_{g \in W_p^r(\XX^{d+1}, W)}
      \left \{ \|f-g\|_{p,W} + t^r\left\|(-\fD_{\g,\mu}^\varrho)^{\f r 2}f \right\|_{p,W} \right \},
$$
where the Sobolev space $W_p^r\big(\XX^{d+1}, W\big)$ is the space that consists of $f \in L^p(\XX^{d+1}, W)$,
even in $t$ variable, so that $\left\|(-\fD_{\g,\mu}^\varrho)^{\f r 2}f \right\|_{p,W}$ is finite. 

For $W_{\g,\mu}^\varrho$, Assertions 1 and 3 hold and we now verify that Assertion 5 holds as well. 
By \eqref{eq:solidHypGdiff}, the kernel $L_n^{(r)}(\varpi)$ in Assertion 5 becomes
$$
    \Lb_n^{(r)}\big(W_{\g,\mu}^\varrho; (x,t),(y,s)\big) =\sum_{k=0}^\infty \wh a\left(\frac{k}{n} \right)
       (k(k+2\g+2\mu+d))^{\f r 2}  \Pb_k^E\big(W_{\g,\mu}^\varrho; (x,t),(y,s)\big).
$$

\begin{lem}
Let $\g, \mu \ge -\f12$ and $\k > 0$. Then, for $r > 0$ and $(x,t), (y,s) \in \XX^{d+1}$, 
$$
 \left| \Lb_n^{(r)}\big(W_{\g,\mu}^\varrho; (x,t),(y,s)\big)\right| \le c_\k   
       \frac{n^{d+r+1}}{\sqrt{ W_{\g,\mu}^\varrho (n; t) }\sqrt{ W_{\g,\mu}^\varrho (n; s) }}
\big(1 + n \sd_{\XX}^\varrho( (x,t), (y,s)) \big)^{-\k}.
$$
\end{lem}

\begin{proof}
This can be proved similarly as in the proof of Lemma \ref{lem:6.71} and using the estimate in
Theorem \ref{thm:kernelXX}. We omit the details. 
\end{proof}

We can now state the characterization of the best approximation by polynomials on the hyperboloid,
following \cite[Theorem 3.12]{X21}. 

\begin{thm}\label{thm:bestappX}
Let $f \in L^p(\XX^{d+1}, W)$ if $1 \le p < \infty$ and $f\in C(\XX^{d+1})$ if $p = \infty$. 
Assume that $f$ satisfies $f(x,t) = f(x,-t)$. Let $\varrho \ge 0$, $r > 0$ and $n =1,2,\ldots$. Then
\begin{enumerate} [   (i)]
\item for $W = W_{\g,\mu}^\varrho$ with $\g,\mu \ge 0$,  
$$
  \Eb_n(f)_{p,W_{\g,\mu}^\varrho} \le c \, \Kb_r (f;n^{-1})_{p,W_{\g,\mu}^\varrho};
$$
\item for every doubling weight $W$ that satisfies $W(x,t) = W(x,-t)$ on $\XX^{d+1}$, 
$$
   \Kb_r(f;n^{-1})_{p,W} \le c n^{-r} \sum_{k=0}^n (k+1)^{r-1}\Eb_k(f)_{p,W}.
$$
\end{enumerate}
\end{thm}
 
For $W = W_{\g,\mu}^\varrho$, the $K$-functional is equivalent to the modulus of smoothness. 

\begin{thm} \label{thm:K=omegaX}
Let $\varrho \ge 0$, $\g,\mu \ge 0$ and $f \in L_p^r(\XX^{d+1}, W_{\g,\mu}^\varrho)$,  $1 \le p \le \infty$. Then
for $0 < \t \le \pi/2$ and $r >0$ 
$$
   c_1 \Kb_r(f; \t)_{p,W_{\g,\mu}^\varrho} \le \o_r(f;\t)_{p,W_{\g,\mu}^\varrho} \le 
   c_2 \Kb_r(f;\t)_{p,W_{\g,\mu}^\varrho}.
$$
\end{thm}

In particular, the characterization in Theorem \ref{thm:bestappX} can be stated in terms of the 
modulus of smoothness instead.

\end{document}